\documentclass[12pt]{amsart}

\usepackage{dsfont}
\usepackage{amsmath}
\usepackage{amssymb}
\usepackage{graphicx}

\usepackage{amssymb}
\usepackage{color}

\usepackage[margin=3cm]{geometry}

\usepackage{amsfonts}

\numberwithin{equation}{section}

\usepackage{color}

\usepackage{paralist}

\usepackage{stmaryrd}

\usepackage{amsxtra}
\usepackage{bbm}

\def\carre{\hfill $\Box$}
\def\one{{\mathbbm{1}}}

\def\supp{{\rm supp}}

\def\tK{\tilde{K}}

\def\th{\theta}
      
\def\trho{{\tilde{\rho}}}

\def\bG{{\boldsymbol{G}}}

\def\bbR{{\mathfrak{R}}}
\def\bbA{{\mathbbm{A}}}
\def\bbG{{\mathbbm{G}}}
\def\bbL{{\mathbbm{L}}}
\def\bbM{{\mathbbm{M}}}

\def\cK{{\mathcal{K}}}

\def\brk{{\bar{k}}}
\def\brK{{\bar{K}}}

\def\brfU{{\bar{\fU}}}
\def\brcL{{\bar\cL}}
\def\breps{{\bar\varepsilon}}
\def\brdelta{{\bar\delta}}

\def\brf{{\bar f}}
\def\bre{{\bar e}}

\newcommand{\bx}{{\bf x}}
\newcommand{\by}{{\bf y}}

\def\cA{\mathcal{A}}
\def\cB{\mathcal{B}}
\def\cG{\mathcal{G}}
\def\cH{\mathcal{H}}
\def\cL{\mathcal L}
\def\cM{\mathcal M}
\def\cP{\mathcal P}
\def\cS{\mathcal S}

\def\cV{\mathcal V}
\def\cZ{\mathcal Z}

\def\eps{\varepsilon}

\def\var{{\text{Var}}}

\def\fa{{\mathfrak{a}}}
\def\fh{{\mathfrak{h}}}
\def\fU{{\mathfrak{U}}}

\def\tcD{\tilde{\cD}}
\def\tLambda{\tilde{\Lambda}}

\def\tcL{\tilde{\cL}}

\def\teps{{\tilde{\eps}}}

\def\E{          \mathbb E}

\def\a{          \alpha}

\def\cF{          \mathcal F}

\def\cA{          \mathcal A}
\def\cD{          \mathcal D}

\def\th{        \theta}
\def\Id{        \rm Id}

\def \R{{\mathbb R}}

\def \Z{{\mathbb Z}}
\def \N{{\mathbb N}}

\newcommand{\T}{{\mathbb T}}
\newcommand{\prf}{{\begin{proof}}}
\newcommand{\epf}{{\end{proof}}}

\newcommand{\bt}{{\mathbf t}}
\newcommand{\ba}{{\mathbf a}}
\newcommand{\bb}{{\mathbf b}}
\newcommand{\bs}{{\mathbf s}}
\newcommand{\br}{{\mathbf r}}

\def\P{{\mathbb P}}

\newtheorem{theo}{\sc Theorem}[section]
\newtheorem{prop}[theo]{\sc Proposition}
\newtheorem{lemm}[theo]{\sc Lemma}
\newtheorem{cor}[theo]{\sc Corollary}
\newtheorem{coro}[theo]{\sc Corollary}

\theoremstyle{definition}

\def\bee{\begin{equation}}
\def\eee{\end{equation}}

\theoremstyle{remark}

\newtheorem{rema}[theo]{\sc Remark}

\newcommand{\pdvr}[2]
{\dfrac{\partial^{#2} #1}{\partial \theta^{#2_1} \partial r^{#2_2}}}
\newcommand{\pdvrs}[2]
{\partial^{#2} #1 /\partial \theta^{#2_1} \partial r^{#2_2}}
\def\brl{{\hat{l}}}

\def\hZ{{\hat{Z}}}
\def\hcD{\hat{\cD}}
\def\hcL{\hat{\cL}}
\def\hPi{\hat{\Pi}}
\def\hxi{\hat{\xi}}
\def\htheta{{\hat\theta}}

\def\tG{{\tilde{G}}}
\def\tbbG{{\tilde{\bbG}}}
\def\tbbL{{\tilde{\bbL}}}
\def\tcM{{\tilde{\cM}}}
\def\tcS{{\tilde{\cS}}}
\def\tu{{\tilde{u}}}
\def\tZ{{\tilde{Z}}}
\def\ttZ{{\tilde{\tZ}}}

\def\teta{{\tilde\eta}}
\def\tsigma{{\tilde\sigma}}
\def\ttheta{{\tilde\theta}}

\newcommand{\diag}{\mathop{\rm diag}}

\def\cN{\mathcal N}

\author{ Dmitry Dolgopyat \and Bassam Fayad\and Ilya Vinogradov}
\title[CLTs for simultaneous Diophantine approximations]
{Central limit theorems for simultaneous Diophantine approximations}

\begin{document}

\begin{abstract}
We study the distribution modulo $1$ 
of the values taken on the integers of 
 $r$  linear forms in $d$ variables with random coefficients.
We obtain quenched and annealed central limit theorems  for the number of simultaneous hits 
into shrinking  targets of radii $n^{-\frac{r}{d}}$. 
By the Khintchine-Groshev theorem on 
Diophantine approximations, $\frac{r}{d}$ is the critical exponent for the infinite number of hits.

\end{abstract}
\maketitle

\section{Introduction}
\label{ScIntro}
\subsection{Results.}
An important problem in Diophantine approximation is the study of
the speed of approach to $0$ of a possibly inhomogeneous
linear form of several variables evaluated at integers points.
Such a linear form is given by $l: \T \times \T^d \times \Z^d \to \T$  
\begin{align} l_{x,\alpha}(k) = x+ \sum_{i=1}^d k_i \a_i \pmod 1
\end{align}
More generally,  one can consider $r$ linear forms for $r\geq 1$,
corresponding to  $\ba := (\alpha_{i}^j)\in \T^{d\times r}$ and $\bx := (x^1,\dots,x^r)\in \T^r$, where each $\alpha^j, j=1,\ldots,r$, is a vector in $\T^d$. 

Diophantine approximation theory classifies the matrices $\ba$ and vectors $\bx$ according to how ``resonant'' they are; i.e.,\ how well the vector $(l_{x^j,\a^j}(k))_{j=1}^r$ approximates ${\bf 0}:=(0,\dots,0)\in \R^r$ as $k$ varies 
over a large ball in $\Z^d.$
One can then fix a sequence of targets converging to $0$, say intervals  of radius $r_n$  centered at $0$ with $r_n \to 0$, and investigate the integers for which the target is {\it hit}, namely the  integers $k$ such that and $l^j_{x^j,\a^j}(k) \in [-r_{|k|},r_{|k|}]$ for every $j=1,\ldots,r$. An important class of targets is given by radii following a power law, $r_n=c n^{-\gamma}$ for some $\gamma,c>0$ (see for example \cite{KH24,Gr,Schmidt} or \cite{BBKM} for a nice discussion related to the Diophantine properties of linear forms). 

Fix a norm $|\cdot|$ on $\R^d$, and let $\|\cdot\|$ denote the Euclidean norm on $\R^r$. For $c>0$ and $\iota=1$ or $2$, we define sets \begin{align}B_\iota(k,d,r,c) =  [0,c|k|^{-\frac{d}{r}}]^r \subset \R^r\end{align}
for
$\iota=1$ 
and 
\begin{align}B_\iota(k,d,r,c) = \{a \in \R^r : \|a\|\le c|k|^{-\frac{d}{r}}\},\end{align}
for $\iota = 2$.
We then introduce 
\begin{align} \label{eq.general} V_{N,\iota}(\ba,\bx,c) 
& =\# \left\{0\leq |k|<N \colon
(l_{x^j,\a^j}(k))_{j=1}^r  \in B_\iota(k,d,r,c) \right\},
\\
 \label{Un}
 U_{N,\iota}(\ba,c) & =V_{N,\iota}(\ba,\bf{0},c). 
\end{align}

A matrix $\ba \in \T^{d\times r}$ is said to be {\it badly approximable} if for some $c>0$, the 
sequence $ U_{N,\iota}(\ba,c)$ is bounded. By contrast, matrices $\ba$ for which  $ U^\eps_{N,\iota}(\ba,c)$ is unbounded, where $U^\eps_{N,\iota}(\ba,c)$ is defined as $U_{N,\iota}$, but with radii $cn^{-\frac{d}{r}-\eps}$ instead of $cn^{-\frac{d}{r}}$
are called {\it very well approximable} or VWA. 
The obvious direction of the Borel-Cantelli lemma 
implies that almost every $\a \in \T^{d\times r}$ is not very well approximable (cf. \cite[Chap. VII]{Cassel}).
The celebrated Khintchine-Groshev theorem on Diophantine approximation implies that badly approximable 
matrices are also of zero measure \cite{Kh, Gr, DS, schmidt0, BV, BBV}. Analogous definitions apply in the inhomogeneous case of $V_{N,\iota}(\ba,\bx,c)$, and similar results hold.

For targets given by a power law, the radii $cn^{-\frac{d}{r}}$ are thus the smallest ones to yield an infinite number of hits almost surely. A natural question is then to investigate statistics of these hits, which we call \emph{resonances}. 
In the present paper we address in this context the behavior of the resonances on average over $\ba$ and $\bx$,  
or on average over $\ba$ while $\bx$ is fixed at ${\bf 0}$ or fixed at random. 
Introduce the ``expected'' number of hits
\begin{align}\widehat{V}_{N,\iota}=\text{Vol} (B_\iota(1, d, r, c)) \ln N \end{align}
when $\bx$ and $\ba$ are uniformly distributed on corresponding tori. Let $\cN(m,\sigma^2)$ denote the normal distribution with mean $m$ and variance $\sigma^2$.  


\begin{theo}
 \label{ThCLT-U} 
Suppose  $(r,d)\neq (1,1),$ 
and let  $\ba$ be uniformly distributed on $\T^{dr}$. Then, 
\begin{align}\frac{U_{N,\iota}(\ba,c)-\widehat{V}_{N,\iota}}{\sqrt{\ln N}}\end{align}
converges in distribution to $\cN(0, \sigma_\iota^2)$ as $N\to\infty$
, where
\begin{align}
\label{Sigma12}
\sigma_1^2& =2 c^rd  \frac{\zeta(r+d-1)}{\zeta(r+d)}\mathrm{Vol}(\cB), & 
\sigma_2^2& =\frac{\pi^{r/2}}{\Gamma(\frac{r}{2}+1)} \sigma_1^2. 
\end{align}
where $\cB$ is the unit ball in $|\cdot|$-norm and Vol denotes the Euclidean volume.
\end{theo}

\begin{rema}
The restriction $(r,d)\neq (1,1)$ above is necessary.    
In fact, it is shown in \cite{Ph, Sam} using continued fractions
that in that case the Central Limit Theorem (CLT) still holds for $U_N$, but the correct normalization should be $\sqrt{\ln N\ln \ln N}$
rather than $\sqrt{\ln N}.$
\end{rema}  

\begin{theo} \label{ThCLT-V} 
Let  $\ba$ and $\bx$ be uniformly distributed in $\T^{d r}$. Then,
\begin{align}\frac{V_{N,\iota}(\ba,\bx,c)-\widehat{V}_{N,\iota}}{\sqrt{\widehat{V}_{N,\iota}}}\end{align}
converges in distribution to $\cN(0,1)$ as $N\to\infty$.
\end{theo}

The preceding theorems give CLTs in the cases of $\bx$ fixed to be $0$ or $\bx$ random. 
The CLT also holds for for almost every fixed $\bx$.   

\begin{theo} \label{ThCLT-V-as} 
Suppose $(r,d)\neq (1,1).$  For almost every $\bx$, 
if  $\ba$ is uniformly distributed in $\T^{d r}$, then 
$\frac{V_{N,\iota}(\ba,\bx,c)-\widehat{V}_{N,\iota}}{\sqrt{\widehat{V}_{N,\iota}}}$  converges in distribution to a normal  random variable with zero mean and variance one.
\end{theo}

\subsection{Plan of the paper.}
Using a by now standard approach of Dani correspondence 
(cf. \cite{Da, M-npoint, M-Mod1, AGT1, AGT2, APT, KSW})
we deduce our results about Diophantine approximations from appropriate limit theorems for homogeneous flows. 
Namely we need to prove a CLT for Siegel transforms of piecewise smooth functions; these limit theorems are formulated
in Section \ref{ScLat}. The reduction of the theorems of Section \ref{ScIntro} to those o  Section \ref{ScLat} is
given in Section \ref{ScDA-Lat}. The CLTs in the space of lattices are in turn deduced
from an abstract Central Limit Theorem (Theorem \ref{ThCLTabs}) for weakly dependent random variables
which is formulated and proven in Section \ref{ScAbstract}. In order to verify the conditions of 
Theorem \ref{ThCLTabs} for the problem at hand we need several results about regularity of Siegel transforms
which are formulated in Section \ref{ScH-SL} and proven in the appendix. In Section \ref{sec.proofs} we deduce our
Central Limit Theorems for homogeneous flows from the abstract Theorem \ref{ThCLTabs}.
Section \ref{sec.variances} contains the proof of the formula \eqref{Sigma12} on the variances. 
Section~\ref{ScGenGroup} discusses some applications of Theorem \ref{ThCLTabs} beyond the subject of Diophantine
approximation.

\subsection*{Acknowledgements}  D. D. was partially supported by the NSF.
I.V. appreciates the support of Fondation Sciences Math\'e\-matiques de Paris during his stay in Paris.

\section{Central Limit Theorems on the space of lattices.}
\label{ScLat}

\subsection{Notation}

We let $\bbG=SL_{d+r}(\R), \tbbG= SL_{d+r} (\R)\ltimes \R^{d+r} .$ 
The multiplication rule in $\tbbG$ takes form
$(A, a)(B, b)=(AB, a+Ab)$. We regard $\bbG$ as a subgroup of $\tbbG$ consisting of elements of the form
$(A, 0)$. We let $\bbL$ 
be the abelian subgroup of $\bbG$ consisting of matrices $\Lambda_\ba,$ 
and $\tbbL$ 
be the abelian subgroup of $\tbbG$ consisting of matrices $(\Lambda_\ba, (0, y))$ where $0$ is an origin in $\R^d,$
$y$ is an $r$-dimensional vector,
$\ba$ is a $d\times r$ matrix 
and  
\[ \Lambda_\ba=
\begin{pmatrix} \text{Id}_d & 0 \\ \ba & \text{Id}_r\end{pmatrix}.
\]
Let $\cM$ be the space of $d+r$ dimensional unimodular lattices and $\tcM$ be the space of $d+r$ dimensional unimodular affine lattices. We identify $\cM$ and $\tcM$ respectively with $\bbG/SL_{d+r}(\Z)$ and  $\tbbG/ SL_{d+r} (\Z)\ltimes \Z^{d+r}$.

We will need spaces $C^{\bs, \br}(\R^p),$  $C^{\bs, \br}(\cM)$,  and $C^{\bs, \br}(\tcM)$ 
of functions which can be well approximated by smooth functions, given $\bs, \br \geq 0$. Recall first that the space $C^{\bs}(\R^p)$ consists of functions $f\colon \R^p \to\R$ whose derivatives up to order $\bs$ are bounded. To define spaces $C^{\bs}(\cM)$ and $C^{\bs}(\tcM)$, fix bases for $\text{Lie}(\bbG)$ and $\text{Lie}(\tbbG)$; then, $C^{\bs}(\cM)$ and $C^{\bs}(\tcM)$ consist of functions whose derivatives corresponding to monomials of order up to $\bs$  in the basis elements are bounded (see Appendix for formal definitions). Now we define $C^{\bs, \br}$-norm on a space equipped with a $C^{\bs}$-norm by
\begin{align}
\|f\|_{C^{\bs,\br}} = \sup_{0 < \eps < 1} \sup_{\substack { f^- \le f \le f^+\\ \|f^+ - f^-\|_{L^1} <\eps}} \eps^{\br} (\|f^+\|_{C^{\bs}} + \|f^-\|_{C^{\bs}}).\label{Cs-L1}
\end{align}
Some properties of these spaces are discussed in the
Appendix.

Given a function $f$ on $\R^{r+d}$ we consider its Siegel transforms
$\cS:\cM\to\R$ and $\tcS:\tcM\to\R$ defined by
\begin{align}\cS(f)(\cL) & =\sum_{e\in \cL} f(e), &
 \tcS(f)(\tcL)& =\sum_{e\in \tcL} f(e). \end{align}
 We emphasize that Siegel transforms of smooth functions are
 never bounded but the growth of their norms at infinity is
 well understood, see Subsection \ref{sec.truncation}.

 \subsection{Results for the space of lattices}

 In this section we present general Central Limit Theorems for Siegel transforms.
 The reduction of Theorems \ref{ThCLT-U}, \ref{ThCLT-V}, \ref{ThCLT-V-as} to the results
 stated here will be performed in Section \ref{ScDA-Lat}.

 Let $f\in C^{\bs, \br}(\R^{d+r})$ be a non-negative function supported on a compact set which does not contain $0$.
 (The assumption that $f$ vanishes at zero is needed to simplify the formulas for the moments of its
   Siegel transform. See Proposition \ref{LmRogers}.)
Denote $\brf=\iint_{\R^{d+r}} f(x,y) dx dy.$

We say that a subset $S\subset \cM$ is \emph{$(K, \alpha)$-regular} if $S$ is a union of codimension 1 submanifolds
and there is a one-parameter subgroup $h_u\subset \bbL$ such that
$$ \mu(\cL: h_{[-\eps, \eps]} \cL\cap S\neq \emptyset)\leq K\eps^\alpha.$$
We say that a function $\rho\colon \cM \to \R$ is \emph{$(K, \alpha)$-regular} if $\supp(\rho)$ has a $(K, \alpha)$-regular
boundary and the restriction of $\rho$ on $\supp(\rho)$ belongs to $C^\alpha$ with
$$ \| \rho \| _{C^\alpha(\supp(\rho))}\leq K.$$
 $(K, \alpha)$-regular functions on $\tcM$ are defined similarly.

Let $\bbA$ be subgroup of $\bbG$ consisting of diagonal matrices. We use the notation $da$ for Haar measure on $\bbA$. 
We say that $\rho$ is \emph{$K$-centrally smoothable} if there is a positive function $\phi$ supported in a unit
neighborhood of the identity in  $\bbA$
such that  $\int_{\mathbb{A}} \phi(a) da=1$ and 
$$ \rho_\phi(\cL):=\int_{\mathbb{A}} \rho(a \cL) \phi(a) da $$
has $L^\infty$ norm less than $K.$ 
We say that a function $\rho$ on $\tcM$ is \emph{$K$-centrally smoothable} if
  $$\rho^*(\cL)=\sup_{\bx} \rho(\cL,\bx)$$ is a $K$-centrally smoothable function on $\cM.$
As before, we write $\cN(m, \sigma^2)$ for the normal distribution with $m$ and variance $\sigma^2$
and
``$\implies$'' stands for convergence in distribution. We write $g$ for a certain diagonal element of $\bbG$ and $\tbbG$ formally introduced in \eqref{eq:def_g}. 


\begin{theo}
 \label{ThCLT-U-Lat} 
Suppose that $(r,d)\neq (1,1).$  

(a) There is a constant $\sigma$ such that if $\cL$ is uniformly distributed on
$\cM$ then
$$ \frac{\sum_{n=0}^{N-1} \cS(f)(g^n \cL)-N\brf}{\sigma \sqrt{N}}  \implies \cN(0,1) $$
as $N\to\infty$.

(b) Fix constants $C,u, \alpha, \eps$ with $\eps<\frac{1}{2}.$
Suppose that $\cL$ is distributed according to
a density $\rho_N$ which is $(CN^u, \alpha)$-regular and $C$-centrally smoothable. Then
$$ \frac{\sum_{n=N^\eps}^{N-1} \cS(f)(g^n \cL)-N\brf}{\sigma \sqrt{N}}  \implies \cN(0,1)   $$
as $N\to\infty$. 
\end{theo}

\begin{theo}
 \label{ThCLT-V-Lat} 
(a) Let $\tcL$ be uniformly distributed on $\tcM.$ 
Then there is a constant $\tsigma$ such that 
$$ \frac{\sum_{n=0}^{N-1} \tcS(f)(g^n \tcL)-N\brf}{ \tsigma \sqrt{N}}  \implies \cN(0,1)  $$
as $N\to\infty$. 

(b) Fix constants $C,u, \alpha, \eps$ with $\eps<\frac{1}{2}.$
Suppose that $\tcL$ is distributed according to
a density $\rho_N$ which is $(CN^u, \alpha)$-regular and $C$-centrally smoothable. Then
$$ \frac{\sum_{n=N^\eps}^{N-1} \tcS(f)(g^n \tcL)-N\brf}{\tsigma \sqrt{N}}  \implies \cN(0,1)   $$
as $N\to\infty$. 
\end{theo}

Let $\tcD$ be an unstable rectangle, that is 
$$\tcD=\{(\Lambda_\ba, (0, \bx))\tcL_0\}_{(\ba, \bx)\in \bbR_1\times \bbR_2} $$
where $\tcL_0$ is a fixed affine lattice and $\bbR_1$ and $\bbR_2$ are rectangles
in $\R^{d\times r}$ and $\R^r$ respectively.  Consider a partition $\Pi$ of
$\tcD$ into $\bbL$-rectangles. Thus elements of $\Pi$ are of the form
$$ \{(\Lambda_\ba, (0, \bx_0))\tcL_0\}_{\ba\in \bbR_1} $$
for some fixed $\bx_0.$

\begin{theo} \label{ThCLT-V-Lat-as} 
Suppose that $(r,d)\neq (1,1).$ Then for each unstable cube $\tcD$ and for almost every ${\bar \cL}\in \tcD$, 
if $\tcL$ is 
uniformly distributed in $\Pi(\brcL)$, then
$$ 
\frac{\sum_{n=0}^{N-1} \tcS(f)(g^n \tcL)-N\brf}{\tsigma \sqrt{N}}  \implies \cN(0,1)$$
as $N\to\infty$. 
\end{theo}


The explicit calculation of
$\sigma$ and $\tsigma$ when $f$ is an indicator functions (the case needed for 
Theorems \ref{ThCLT-U}-\ref{ThCLT-V-as})
will be given in Section \ref{sec.variances}. 



\begin{rema}
Central Limit Theorems for partially hyperbolic translations on homogeneous spaces are proven
in \cite{D} (for bounded observables) and in \cite{LB} (for $L^4$ observables).
(See also \cite{S, Rat} for important special cases).
It seems possible
to prove Theorem \ref{ThCLT-U-Lat} for sufficiently large values of $d+r$ by verifying the conditions
of \cite{LB}. Instead, we prefer to present in the next section an abstract result which will later be
applied to derive Theorems \ref{ThCLT-U-Lat}, \ref{ThCLT-V-Lat}, and \ref{ThCLT-V-Lat-as}.
We chose this approach for three reasons. First, this will make the paper self contained.
Second, we replace the $L^4$ assumption of \cite{LB} by a weaker $L^{2+\delta}$ assumption
which is important for small $d+r.$ Third, our approach allows to give a unified proof for
Theorems \ref{ThCLT-U-Lat}, \ref{ThCLT-V-Lat}, and~\ref{ThCLT-V-Lat-as}.  
\end{rema}

\section{Diagonal actions on the space of lattices and Diophantine approximations.} 
\label{ScDA-Lat}

In this section we reduce Theorem \ref{ThCLT-U},  \ref{ThCLT-V},  
and \ref{ThCLT-V-as}
to Theorems \ref{ThCLT-U-Lat}--\ref{ThCLT-V-Lat-as}.

To fix our notation we consider $U_{N,1}$ and $V_{N,1},$
the analyis of $U_{N,2}$ and $V_{N,2}$ being similar.
We also drop the extra subscript and write
$U_{N,1}$ and $V_{N,1}$ as $U_N$ and $V_N$, respectively, until the end of this section.

In Subsections \ref{sec.dani} and \ref{sec.measure} we explain how to reduce Theorem \ref{ThCLT-U}
to Theorem \ref{ThCLT-U-Lat}. The reductions of Theorems \ref{ThCLT-V} and \ref{ThCLT-V-as}
to Theorems \ref{ThCLT-V-Lat} and \ref{ThCLT-V-Lat-as} require only minor modifications which will be
detailed in Section \ref{SSRed-V}.

\subsection{Dani correspondence.}
\label{sec.dani}
In this subsection we use the Dani correspondence principle to reduce the problem to a CLT for 
the action of diagonal elements on the space of lattices of the form $\Lambda_{\bf a}$ where  $\ba$ is random.

Let $\ba$ be the matrix with rows $\a_i \in \R^d$, $i=1,\ldots,r$. 
For $p \in \N$ and $t \in \R$, we denote the $p\times p$ diagonal matrix 
\begin{align}D_p(t)= 
 \begin{pmatrix}
  2^t\\
  & \ddots\\
  && 2^t
 \end{pmatrix}.
\end{align}
We then consider the following matrices
\begin{align}
g & =
\begin{pmatrix}
 D_d(-1) & 0\\ 0 &D_r(\frac dr)
\end{pmatrix},
&
\Lambda_{\ba} & = 
\begin{pmatrix}
 \text{Id}_d &0\\
 \ba & \text{Id}_r
\end{pmatrix}.
\label{eq:def_g}
\end{align}
Let $\phi$ be the indicator of the set 
\begin{equation} \label{eq.Ec}
  E_c:=\left\{
  (x,y) 
  \in\R^{d}\times\R^r\mid
  |x| \in [1,2],
  |x|^{d/r} y_j \in [0,c] \text{ for }j=1,\dots,r\right\}
\end{equation} 
and consider its Siegel transform $\Phi=\cS(\phi)$.

Now $n=(n_1,\ldots,n_d)$ with $|n|\le N$ contributes to $U_N(\ba,c)$
(from \eqref{Un}) precisely when there exists $(m_1,\ldots,m_r)\in \Z^r$ such that 
\begin{equation} \left(\sum_{i=1}^d n_i \a_{1,i} +m_1,\ldots,\sum_{i=1}^d n_i \a_{r,i} +m_r\right)  \in B(n,d,r,c) \label{hit1}. \end{equation}
Clearly such a vector $(m_1,\ldots,m_r)$ is unique. 
%
It is elementary to see that \eqref{hit1} holds if and only if 
\begin{equation}  g^t \Lambda_\ba (n_1,\ldots,n_d,m_1,\ldots,m_r) \in E_c\label{hit2}. \end{equation}
for some integer $t \le 2^{[\log_2 N]}$.

From \eqref{hit1}--\eqref{hit2} we obtain  
\begin{lemm}
\label{LmRed1}
  For each $\eps>0$
\[U_N(\ba,c) = \sum_{t=[(\log_2 N)^\eps]}^{[\log_2 N]} \Phi\circ g^t(\Lambda_\ba \Z^{d+r})+R_N,\] 
with 
\[{\left\| R_N \right\|_{L^1(\ba)}} \ll (\log_2 N)^\eps\] 
and $L^1(\ba)$ denoting the $L^1$-norm with respect to the Lebesgue measure on the unit cube in
$\R^{d\times r}.$
  \end{lemm}

\begin{proof}
  From \eqref{hit1}--\eqref{hit2}, it follows that 
\[R_N\le  \# \left\{ 
\begin{array}{cc}2^{[\log_2 N]} \leq |n|<N,\\ \text{or } |n|\leq 2^{\log_2^\eps N}
\end{array}:
  (l_1(n,0,\a_1), \dots, l_r(n,0,\a_r))   \in B(n,d,r,c) \right\} .\]
  Note that for a fixed $n\in \Z^d$ and $j\in 1,\dots ,r$, the form $\{l_j(n, 0, \a_j)\}$
  is uniformly distributed on the circle.
Hence
$$ \text{Leb}\left\{\ba\in \T^{rd}: (l_1(n,0,\a_1), \dots, l_r(n,0,\a_r))   \in B(n,d,r,c)\right\}\ll \frac{1}{|n|^d}  $$
and so
$$\left\| R_N \right\|_{L^1(\ba)} \ll
\sum_{2^{[\log_2 N]}\le |n|\le N} \frac{1}{|n|^d}+\sum_{|n|\le 2^{\log_2^\eps N}} \frac{1}{|n|^d}
\ll (\log_2 N)^\eps.
$$
\end{proof}

Hence, to prove Theorem \ref{ThCLT-U} we can replace $U_N(\ba,c)$ by
$$ \sum_{t=[(\log_2 N)^\eps]}^{[\log_2 N]} \Phi\circ g^t(\Lambda_\ba \Z^{d+r}). $$ 


\subsection{Changing the measure.} \label{sec.measure}
Note that the action of $g^t$ on the space of lattices $\cM$ is partially hyperbolic and 
its unstable manifolds are orbits of the action $\Lambda_\ba$ with $\ba \in M(d,r)$ ranging in the set of $d\times r$ matrices.
This will allow us to reduce the proof of CLTs to CLTs for the diagonal action
on the space of lattices.
A similar reduction is possible for the $g^t$-action on the space of affine lattices  since in that case  the
unstable manifolds for the action of $g^t$ are given
by $(\Lambda_\ba, (0, \bx))$, with  $\ba \in M(d,r)$, $\bx \in \R^r.$

Let $\eta=1/k^{10d}$ where $k=[\log_2 N]$.
For $i=1,\ldots,r+d-1$ let $t_i$ be independent uniformly distributed
in $[-\eta,\eta]$. Also introduce a random matrix $\bb \in M(r,d)$
where all the entries of $\bb$ are 
independent uniformly distributed in $[-1,1]$. Let 
$$D_{\bt}=
\diag\left(1+t_1,\ldots,1+t_{d+r-1},\left(\textstyle\prod_{l=1}^{r+d-1} (1+t_l)\right)^{-1}\right) $$
and 
$$\bar{\Lambda}_\bb=
 \begin{pmatrix}
  \Id_d & \bb\\ 0 & \Id_r
 \end{pmatrix}
 . $$
Let $\tilde{\Lambda}(\ba, \bb, \bt)=D_\bt \bar{\Lambda}_\bb \Lambda_\ba$. 

It is clear that if $\ba$ is distributed uniformly in a unit cube,
then
$\tilde{\Lambda}(\ba, \bb, \bt)$ is distributed according to a
$(C k^{10d}, 1)$-regular and $C$-centrally smoothable density.
Note that
$$ g^t \tilde{\Lambda}(\ba, \bb, \bt)=D_\bt \bar{\Lambda}_{\bb_t} g^t \Lambda_\ba
\text{ where } |\bb_t|\leq e^{-t}. $$
Observe also that for $h \in \bbG$  and $E\subset \R^{d+r}$, we have
$$ \cS(\one_E)(h \cL)=\cS(\one_{h^{-1}E})(\cL) ,$$
and hence for $t\geq k^{\eps}$ we have 
$$ \left|\cS(\one_{E_c})(g^t \tilde{\Lambda}(\ba, \bb, \bt)\Z^{d+r})-
\cS_{\one_{E_c}}(g^t \Lambda_\ba\Z^{d+r})\right|\leq
\cS_{\one_{\tilde{E}}}(g^t \Lambda_\ba\Z^{d+r}) $$
where $\tilde{E}$ denotes a $C k^{-10 d}$ neighborhood of the boundary of $E_c.$  
Now the same argument as in Lemma \ref{LmRed1} gives

\begin{lemm}
\label{LmRed2}
  For each $\eps>0$, 
  \[\sum_{t=[(\log_2 N)^\eps]}^{[\log_2 N]} \Phi\circ g^t(\Lambda_\ba \Z^{d+r})=
  \sum_{t=[(\log_2 N)^\eps]}^{[\log_2 N]} \Phi\circ g^t(\tilde{\Lambda}(\ba, \bb,\bt) \Z^{d+r})
  +\tilde{R}_N\] 
with 
${\big\| \tilde{R}_N \big\|_{L^1(\ba, \bb, \bt)}} \ll 1 .$
  \end{lemm}

Now Theorem \ref{ThCLT-U} follows from Theorem \ref{ThCLT-U-Lat} except for the formula 
for $\sigma$ which is derived in Section \ref{sec.variances}.

\subsection{Inhomogeneous case}
\label{SSRed-V}

The reduction of Theorems \ref{ThCLT-V} and \ref{ThCLT-V-as} to
Theorems \ref{ThCLT-V-Lat} and \ref{ThCLT-V-Lat-as}
requires only small changes compared to the preceding section. 
To wit, Lemmas \ref{LmRed1} and \ref{LmRed2} take the following form.

\begin{lemm}
\label{LmRedV}
(a) For each $\eps>0$,
\[V_N(\ba,\bx, c) = \sum_{t=[(\log_2 N)^\eps]}^{[\log_2 N]} \tcS(\one_{E_c})(g^t(\Lambda_\ba \Z^{d+r}+(0, \bx)))
+R_N\] 
where 
${\left\| R_N \right\|_{L^1(\ba)}} \ll (\log_2 N)^\eps.$ 

(b) Let $\bb$ and $\bt$ have the same distribution as in Subsection \ref{sec.measure} and $\by$ be uniformly
distributed in $[-1, 1]^d$
then for each $\eps>0$
\begin{multline}\sum_{t=[(\log_2 N)^\eps]}^{[\log_2 N]} \tcS(\one_{E_c})(g^t(\Lambda_\ba \Z^{d+r}+(0, \bx)))
\\
=\sum_{t=[(\log_2 N)^\eps]}^{[\log_2 N]} \tcS(\one_{E_c})(g^t(\tLambda(\ba, \bb, \bt) \Z^{d+r}+(\by, \bx)))
   +\tilde{R}_N \end{multline}
with 
${\big\| \tilde{R}_N \big\|_{L^1(\ba, \bb, \bt, \by)}} \ll 1 .$
\end{lemm}

Note that in part (a) the error has small $L^1(\ba)$ norm for each fixed $\bx.$ This follows 
from the fact that for each $\bx$ and $k$, $\ba k+\bx$ is uniformly distributed on $\T^r.$ 
This is useful in the proof of Theorem \ref{ThCLT-V-as} since we want to have a control 
for each (or at least, most) $\bx.$ We also note that part (b) is only needed for Theorem
\ref{ThCLT-V} since in Theorem \ref{ThCLT-V-Lat-as} we start with initial conditions supported 
on a positive codimension submanifold of $\tcM.$ (One of the steps in the proof of Theorem
\ref{ThCLT-V-Lat-as} consists of fattening the support of the initial measure, see 
Subsection \ref{ScasCLT}, however Lemma \ref{LmRedV}(b) is not needed at the reduction stage).

\section{An abstract Theorem}
\label{ScAbstract}
\subsection{The statement}


In what follows $C,u>0,\theta\in(0,1)$, and $s>2$  are fixed constants. 
Let $\xi_n$ be a sequence of random variables satisfying the following conditions. Write $\hat{\xi_l} = \xi_l - \E(\xi_l)$ for the corresponding centered random variable.  


\begin{enumerate} 


\item[(H1)] Given any $K$, there is a sequences $\xi_n^K$ of random variables 
such that 
\begin{enumerate}
\item[(H1a)] $|\xi_n^K|\leq K$ almost surely;
\item[(H1b)] $\E\left(| \xi_n^K-\xi_n|\right)=O(K^{1-s}),$ 
$\E((\xi_n^K-\xi_n)^2)=O(K^{2-s}).$ 
\end{enumerate}


\item[(H2)] There exists a filtration $\cF_l=\cF_{l, n}$ defined for $0\leq l<n$ 
such that for every $l,k$ there exists a variable $\xi_{l,l+k}^K$ that is   $\cF_{l+k}$-measurable, $|\xi_{l,l+k}^K|\le K$ almost surely, $ \E \xi_{l,l+k}^K = \E \xi_{l}^K$, and 
$$\P( |\xi_{l}^K - \xi_{l,l+k}^K| \geq \theta^k) \leq C K^u (l+1)^u \theta^k.$$

\item[(H3)]  For $l,k$, there exists an event $G_{l,k}$ such that $\P(G_{l,k}^c) \leq C \theta^k$ and 
for $\omega\in G_{k,l}$ 
$$|\E (\hxi_{l+k}^K | \cF_l)(\omega)  |\leq C (l+1)^u K^u  \theta^k.$$
\item[(H4)] 
There exists a numerical sequence $b_{K,k}$ for $k\ge 0$ such that for $\omega\in G_{k,l}$  and $k'\geq k$,
$$\left|\E (\hxi_{l+k}^K\hxi_{l+k'}^K | \cF_l )(\omega) - b_{K, k'-k}\right| \leq 
C(l+1)^u K^u e^{u(k'-k)} \theta^{k}.$$
\end{enumerate}

\begin{theo}
\label{ThCLTabs}
(a) Under conditions (H1)-(H4) if $\omega$ is distributed according to $\P$ then
\begin{equation} \label{clt1} \frac{ \sum_{l=0}^{n-1} \hxi_l  }{\sigma \sqrt{n}} \implies \cN(0,1)  \end{equation}
with 
\begin{equation} \label{eq.sigma}  \sigma^2=\sigma_0+2\sum_{j=1}^\infty \sigma_j \qedhere, \quad \sigma_j=\lim_{K\to\infty} b_{K, j} .\end{equation}

(b) Suppose conditions (H1)--(H4) are satisfied. Fix $\eps > 0$ such that $\frac{1+\eps}{s} + \eps <\frac 1 2$ and set $K_n=n^{\frac{1+\eps}{s}}$. Suppose that $\omega$ is distributed according to a measure
$\P_n$ which has a density $\rho_n=\frac{d\P_n}{d\P}$ satisfying

(D1) $\rho_n\leq C n^u;$

(D2) for each $k$ there is an $\cF_k$-measurable density $\rho_{n,k}$ such that
$$\P(|\rho_n-\rho_{n,k}|\geq \theta^k) \leq  C n^u \theta^k;$$

(D3) For each $n^\eps\leq l\leq n$, 
$$\E_n(|\xi_l - \xi_l^{K_n}|)\leq  C K_n^{1-s}.$$

Then, 
\begin{equation} \label{clt1a} \frac{ \sum_{l=n^\eps}^{n-1} \hxi_l  }{\sigma \sqrt{n}} \implies \cN(0,1).  \end{equation}
\end{theo}

\subsection{Limiting variance.}
Here we show that the normalized variance converges. 

\begin{lemm} Under conditions (H1)--(H4) we have that 
\begin{equation}
\label{GreenKubo}
\lim_{n \to \infty} \frac{\E(( \sum_{l=0}^{n-1} \hxi_l )^2)}{n} = \sigma^2
\end{equation}
with $\sigma$ as in \eqref{eq.sigma}. 
\end{lemm}

\begin{proof}
First we record a property of cross-terms in the sum that lets us pass from $\xi_n$ to the truncated sequence. We have
 \begin{align}
\label{DifVar} D:= \E(\hxi_{m+j} \hxi_m)-\E(\hxi_{m+j}^K \hxi_m^K)=O(K^{2-s}). 
\end{align}
Indeed, we have
\begin{align*}
 D & = 
 \E(\xi_{m+j} \xi_m)-\E(\xi_{m+j}^K \xi_m^K)\\
 &\phantom{=}-\E(\xi_{m+j})\E( \xi_m)+\E(\xi_{m+j}^K)\E( \xi_m^K)\\
 &= \E(\xi^K_{m+j}(\xi_m-\xi_m^K)) + \E(\xi_m^K(\xi_{m+j} - \xi_{m+j}^K)) + \E((\xi_m-\xi_m^K)(\xi_{m+j} - \xi_{m+j}^K))\\
 &\phantom{=}-\E(\xi^K_{m+j})\E(\xi_m-\xi_m^K) - \E(\xi_m^K)\E(\xi_{m+j} - \xi_{m+j}^K) - \E(\xi_m-\xi_m^K)\E(\xi_{m+j} - \xi_{m+j}^K).
\end{align*}
Now applying the Cauchy-Schwarz inequality followed by (H1a) and (H1b), we arrive at the bound \eqref{DifVar} since $s>2$. 



Next, applying (H3) and (H4) with $l=0$  
gives
\begin{equation}
\label{B-Cor}
\E(\hxi_{m+j}^K \hxi_m^K)=b_{K,j}+O(K^{u} 2^{uj} \theta^m) + O(K^2 \theta^m).
\end{equation}
Combing \eqref{DifVar} and \eqref{B-Cor} we get 
\begin{equation}
\label{B-Cor2}
 b_{K,j}=\E(\hxi_{m+j}\hxi_m)+O(K^{2-s})+O(K^{u} 2^{uj} \theta^m)+O(K^2 \theta^m).  
\end{equation}
Take a small number $\teps>0$ and assume that $j\leq \teps K,$  $K<\tK<2K.$  
Taking $m=K/2$ we see that
$$  b_{K,j}-b_{\tK, j}=O(K^{2-s}). $$
Therefore for each $j$, the following limits exist,
$$\sigma_j:=\lim_{K\to\infty} b_{K, j}.$$
and moreover
\begin{equation}
\label{B-Sigma}
 b_{K,j}=\sigma_j+O(K^{2-s}). 
 \end{equation}
Next we claim that under 
 the condition
\begin{equation}
\label{mjk}
m+j<K^{s/(1+\eps)},
\end{equation}
there exists $\tilde u >0 $ such that
\begin{equation}
\label{MixK}
\E(\hxi_{m+j}^K \hxi_m^K)=O(K^\tu \theta^{j/2}).  
\end{equation}
Indeed,
$$ \E(\hxi_{m+j}^K \hxi_m^K)=\E(\hxi_{m+j}^K \hxi_{m,m+ j/2}^K)+\E(\hxi_{m+j}^K (\hxi_{m,m+ j/2}^K-\hxi_{m}^K))=I+I\!\!I.$$
Using (H1a) and (H2), we get
$$ |I\!\!I|\leq K \E(|\hxi_{m,m+ j/2}^K-\hxi_{m}^K|)\leq K\, O\left(\theta^{j/2}+ K^{u+1} (m+1)^u \theta^{j/2}\right) $$
while (H3) shows that
$$ |I|=\left|\E\left(\hxi_{m,m+ j/2}^K \E\left(\hxi_{m+j}^K |\cF_{m+j/2}\right)\right)\right|=
O\left(K^2 \theta^{j/2}\right) + O \left((m+j/2+1)^u K^u \theta^{j/2}\right) $$
proving \eqref{MixK}.

Combining \eqref{DifVar} and \eqref{MixK} we see that for $K$ satisfying \eqref{mjk} we have
$$ \E(\hxi_{m+j} \hxi_m)=O(K^\tu \theta^{j/2})+O(K^{2-s}). $$
Take a small $\teps.$ Suppose first that
$$m\leq e^{\teps j}. $$
Then we can take $K=\theta^{-j/2\tu}$ and \eqref{mjk} holds giving
\begin{equation}
\label{ED}
 \E(\hxi_{m+j} \hxi_m)=O(\htheta^j)
\end{equation} 
for some $\hat\theta\in(0,1)$. In particular combining \eqref{B-Cor2}, \eqref{B-Sigma} and \eqref{ED} with $m=2^{\teps j},$
$K=2m$ we obtain

\begin{equation}
\label{EDSigma}
 \sigma_j=O(\htheta^j). 
  \end{equation}
Also \eqref{B-Cor2} and \eqref{B-Sigma} show that under \eqref{mjk}, 
$$ \E(\hxi_{m+j}\hxi_m)=\sigma_j+O(K^{2-s})+O(K^{u} 2^{uj} \theta^m). $$
Choosing $K=2^{\teps m}$ and assuming $m>\frac{2uj}{\lvert \log_2 \theta\rvert}$, 
\begin{equation}
\label{ExpMix}
 \E(\hxi_{m+j}\hxi_m)=\sigma_j+O(\ttheta^m). 
\end{equation}

To prove the Lemma, we need to control the sum 
$$\frac{\E(( \sum_{l=0}^{n-1} \hxi_l )^2)}{n}=\frac{1}{n}
\left[\sum_{m=0}^{n-1} \E(\hxi_m^2)+2\sum_{m=0}^{n-1} \sum_{0<j\leq n-1-m} \E(\hxi_m \hxi_{m+j})\right]. $$
Using \eqref{ExpMix} for $m>\frac{2u j}{\lvert\log_2 \theta\rvert}$  and using \eqref{ED} otherwise we see that
$$ \E\left( \sum_{m=0}^{n-1} \hxi_m \hxi_{m+j}\right)=n\sigma_j+O(j). $$
Combining this with \eqref{EDSigma} we see that the limit in \eqref{GreenKubo} exists and moreover that
$$ \hskip5cm\sigma^2=\sigma_0+2\sum_{j=1}^\infty \sigma_j. \hskip5cm \qedhere $$
\end{proof}

\subsection{Proof of Theorem \ref{ThCLTabs}}

For all the sequel, we fix $n$  and let $K=K_n=n^{\frac{1+\eps}{s} }$. 
Let $l_m= m [n^\eps]$, $m =0,\ldots,[n^{1-\eps}]:=m_n$.    
Denote 
\begin{align}Z_m & = \sum_{l=l_m+n^{\eps^2}}^{l_{m+1}-1} \hxi_l^K, 
&
\tZ_m & = \sum_{l=l_m+n^{\eps^2}}^{l_{m+1}-1} \hxi_l, 
&
\ttZ_m & = \sum_{l=l_{m}}^{l_m+n^{\eps^2}-1} \hxi_l  ,
&
\check{Z} & =\sum_{l=l_{m_n}}^{n-1}   \hxi_l ,
\end{align}
so that 
$$\sum_{l=0}^{n-1}\hxi_l =\sum_{m=0}^{m_n-1} \ttZ_m+\sum_{m=0}^{m_n-1} \tZ_m+\check{Z}. $$
  We claim that 
\begin{equation} \label{smallblocks} \frac{\sum_{m=0}^{m_n-1} \ttZ_m+\check{Z}}{\sqrt{n}}  \quad \text{and}     \quad \frac{\sum_{m=0}^{m_n-1} (Z_m-\tZ_m)}{\sqrt{n}} \end{equation}
converge to $0$ in probability; it will therefore suffice to prove that
 \begin{equation} \label{clt2} \frac{ \sum_{m=0}^{m_n} Z_m  }{\sigma \sqrt{n}} \implies \cN(0,1). \end{equation}
To prove the claim, observe first that following the proof of \eqref{GreenKubo}, we have 
\begin{equation}
\label{SmallBlocksA}
\E \left(    \left(    \sum_{m=0}^{m_n} \ttZ_m  +\check{Z} \right)^2 \right) = O(n^{1-\eps+\eps^2}) =o(n). 
\end{equation}
As for the second sum in \eqref{smallblocks}, we note that 
$$ Z_m-\tZ_m=\sum_{l=l_m+n^{\eps^2}}^{l_{m+1}-1} \left[\left(\xi_l^K-\xi_l\right)-\left(
\E(\xi_l^K)-\E(\xi_l)\right)
\right] $$
hence (H1b) gives
$$ \E|Z_m-\tZ_m|\leq
2 \sum_{l=l_m+n^{\eps^2}}^{l_{m+1}-1} \E(|\xi_l^K-\xi_l|)\leq C n^\eps K^{1-s}. $$
Therefore
\begin{equation}  \sum_m \frac{\E|Z_m-\tZ_m|}{\sqrt{n}}=O\left(K^{1-s} \sqrt{n}\right)
=O\left(n^{\frac{1+\eps}{s}-\frac{1}{2}}\right)=O\left(n^{-\eps}\right)\to 0. \label{truncated} \end{equation}



We turn now to the proof of \eqref{clt2}.

We start by defining an exceptional set $\bG_m^c$ on which we will not be able to exploit the almost independence of $Z_{m+1}$ from $(Z_1,\ldots,Z_m)$.   
The exact reasons for the definition of each condition on $\bG_m$ will appear during the proof. 

Let $\brl_{m+1}=l_{m+1}+n^{\eps^2/2}.$ 
To be able to use (H2)--(H4) we let for $m\leq m_n$
$$G_m^{(1)}=\bigcap_{k=n^{\eps^2}}^{n^\eps} \left(G_{l_{m+1},k}\cap G_{\brl_{m+1}, k}\right). $$
Next define for $l\geq l_{m+1}+n^{\eps^2}$
$$ \tG_l=\left\{\omega: \E\left(\left|\xi_l^K-\xi_{l,l+n^{\eps^2/2}}^K\right|\; \big|\cF_{\brl_{m+1}}\right)(\omega)
\leq \theta^{n^{\eps^2/10}}\right\} $$
and $$G_m^{(2)}=\bigcap_{k=n^{\eps^2}}^{n^\eps} \tG_{l_{m+1}+k}.$$
For $k,k' \in [n^{\eps^2},n^{\eps}]$ with     $ k'-k \geq n^{\eps^2/2} $ define 
 $$E_{m,k,k'}=\left\{\omega' : \left|\E( \hxi_{l_{m+1}+k'}^K |\cF_{l_{m+1}+k+n^{\eps^2/2}})(\omega') \right|\geq \theta^{n^{\eps^2/10}} \right\}$$
 and let $$\bar{G}_{m,k,k'} = \left\{ \omega : \E \left( \one_{E_{m,k,k'}}(\omega') | \cF_{\brl_{m+1}}   \right) (\omega) \leq  \theta^{n^{\eps^2 /10}} \right\}$$
and 
$$G_m^{(3)}=\bigcap_{k,k'     \in [n^{\eps^2},n^{\eps}] :  k'-k\geq n^{\eps^2/2} }  {\bar G}_{m,k,k'}.$$
Finally set 
$$\bG_m=G_m^{(1)} \cap G_m^{(2)} \cap G_m^{(3)}. $$
Observe that (H2)--(H4) show that
\begin{equation}
\label{GmCompl}
\P(\bG_m^c) \leq C \theta^{n^{\eps^2/100}}.
\end{equation}
The main step in the proof of the CLT is the following 
\begin{lemm} \label{lemma.clt3}
  For $\omega\in \bG_m$,
\begin{equation}
\label{clt3} 
\ln \E\left(e^{i \lambda \frac{Z_{m+1}}{\sqrt{n}} } | \cF_{l_{m+1}}\right)(\omega)  = -\frac{n^\eps}{n} \lambda^2 \sigma^2 (1+o(1)) \end{equation}
where $o(1)$ is uniform in $m=1,\ldots,m_n$. 
\end{lemm}

\begin{proof} 
To prove \eqref{clt3} we note that
\begin{align}\E\left(e^{i \lambda  \frac{Z_{m+1}}{\sqrt{n}} } | \cF_{\brl_{m+1}}\right)(\omega) &=  
1+  i \frac{\lambda}{\sqrt{n}}  \E(Z_{m+1} | \cF_{\brl_{m+1}} )(\omega) -
 \frac{\lambda^2}{2{n}}  \E(Z_{m+1}^2 | \cF_{\brl_{m+1}} )(\omega) \nonumber \\
&  \phantom{=} +   O \left(\frac{\lambda^3}{{n}^{\frac 3 2}}  \E(Z_{m+1}^3 | \cF_{\brl_{m+1}} )(\omega) \right)  \nonumber   \\ 
&=   1+  i \frac{\lambda}{\sqrt{n}}  \E(Z_{m+1} | \cF_{\brl_{m+1}} )(\omega) -
 \frac{\lambda^2}{2{n}}  \E(Z_{m+1}^2 | \cF_{\brl_{m+1}} )(\omega)  \nonumber  \\
& \phantom{=}+ o\left( \frac{\lambda^2}{2{n}}  \E(Z_{m+1}^2 | \cF_{\brl_{m+1}} )(\omega) \right) \label{clt.dl}
  \end{align}
where the last step uses that $|Z_{m+1}|\leq Kn^\eps =o(n^{\frac1 2})$.

Next, (H3) implies that on $\bG_m$
$$\E(Z_{m+1} | \cF_{\brl_{m+1}} )(\omega)=o\left(\theta^{0.5 n^{\eps^2}}\right). $$ 
To finish the proof of \eqref{clt3} it suffices to show that for $\omega\in \bG_m$,
\begin{equation}
\label{BlockVar}
|\E(Z_{m+1}^2 | \cF_{\brl_{m+1}} )(\omega) - n ^\eps \sigma^2 |=o(n^\eps).
\end{equation}
Note that
$$\E(Z_{m+1}^2 | \cF_{\brl_{m+1}} )(\omega)=\sum_{k, k'} 
\E(\hxi_{l_{m+1}+k}^K \hxi_{l_{m+1}+k'}^K | \cF_{\brl_{m+1}} )(\omega). $$
Let us estimate the individual terms in this sum. To fix our notation let us suppose that $k'\geq k.$
Let $R$ be a large constant and consider two cases. 

(a) $k>R (k'-k).$ In this case (H4) and \eqref{B-Sigma} give
$$ \E (\hxi_{l_{m+1}+k}^K \hxi_{l_{m+1}+k'}^K | \cF_{\brl_{m+1}} )(\omega)+O\big(\ttheta^k\big)
=b_{K, k'-k}=\sigma_{k'-k}+O\left(K^{2-s}\right)+O\big(\ttheta^k\big). $$ 

(b) $k\leq R(k'-k)$ and hence $k'-k>\frac{n^{\eps^2}}{R}.$ Then
\begin{align}
 \E (\hxi_{l_{m+1}+k}^K \hxi_{l_{m+1}+k'}^K | \cF_{\brl_{m+1}} )(\omega)
 &=
 \E (\hxi_{l_{m+1}+k, l_{m+1}+k+n^{\eps^2/2} }^K \hxi_{l_{m+1}+k'}^K | \cF_{\brl_{m+1}} )(\omega) \\
 &\phantom{=}
 +
\E \left(\left(\hxi_{l_{m+1}+k}^K-\hxi_{l_{m+1}+k,  l_{m+1}+k+n^{\eps^2/2} }^K\right) \hxi_{l_{m+1}+k'}^K | \cF_{\brl_{m+1}} \right)(\omega)\\
&
=I+I\!\!I 
\end{align}
The second term is
$ O\left(\theta^{n^{\eps^2/10}} K\right)$ since $\omega\in\tG_{l_{m+1}+k}.$
For the first term use that $\omega\in \bar{G}_{m,k, k'}$ to obtain

\begin{align*} \left|I\right|
&=
\left|\E (\hxi_{l_{m+1}+k,  l_{m+1}+k+ n^{\eps^2/2} }^K \E( \hxi_{l_{m+1}+k'}^K |\cF_{l_{m+1}+k+n^{\eps^2/2}})
| \cF_{\brl_{m+1}} )(\omega)\right|
\\ &
\leq K^2  \E ( \one_{E_{m,k,k'}}(\omega') | \cF_{\brl_{m+1}})(\omega) + K \theta^{n^{\eps^2/10}}\\
&
\leq K^2 \theta^{n^{\eps^2/10}}\end{align*}
so both $I$ and $I\!\!I$ are negligible. Combining the estimates of cases (a) and (b) we obtain \eqref{BlockVar}
completing the proof of \eqref{clt3}. \end{proof}

To finish the proof of part (a) it remains to derive the Central Limit Theorem from \eqref{clt3}. For $j\leq m$
set
$$ \hZ_j= \sum_{l=l_j+n^{\eps^2}}^{l_{j+1}} \hxi_{l, n^{\eps^2/2}}^K. $$
Then 
$$ \E\left(\sum_{j=0}^m \big|\hZ_j-Z_j \big|\right)=O\left(\theta^{n^{\eps^2/10}}\right)$$ 
and so
\begin{align}
\label{Z-hZ1}
\E\left(e^{i\frac{\lambda}{\sqrt{n}} \sum_{j=0}^{m+1} Z_j}\right)-\E\left(e^{i\frac{\lambda}{\sqrt{n}} (\sum_{j=0}^{m} \hZ_j)+i\frac{\lambda}{\sqrt{n}}Z_{m+1}}\right)
&=
O\left(\theta^{n^{\eps^2/10}}\right),
\\
\label{Z-hZ2}
\E\left(e^{i\frac{\lambda}{\sqrt{n}} \sum_{j=0}^{m} Z_j}\right)-\E\left(e^{i\frac{\lambda}{\sqrt{n}} \sum_{j=0}^{m} \hZ_j}\right)
&=
O\left(\theta^{n^{\eps^2/10}}\right).
\end{align}
Therefore,
\begin{align}\E\left(e^{i\frac{\lambda}{\sqrt{n}} \sum_{j=0}^{m+1} Z_j}\right)
&\overset{\eqref{Z-hZ1}}{=}
\E\left(e^{i\frac{\lambda}{\sqrt{n}} \sum_{j=0}^{m} \hZ_j}
\E\left(e^{i\frac{\lambda}{\sqrt{n}} Z_{m+1}}|\cF_{\brl_{m+1}}\right)\right)+O\left(\theta^{n^{\eps^2/10}}\right)
\\
&
\overset{\eqref{GmCompl}}{=}\E\left(e^{i\frac{\lambda}{\sqrt{n}} \sum_{j=0}^{m} \hZ_j}
\one_{\bG_m} \E\left(e^{i\frac{\lambda}{\sqrt{n}} Z_{m+1}}|\cF_{\brl_{m+1}}\right)\right)+O\left(\theta^{n^{\eps^2/100}}\right)
\\
&
\overset{\eqref{clt3}}{=}e^{-\frac{\sigma^2 {\lambda}^2 n^\eps}{2 n}} \E\left(e^{i\frac{\lambda}{\sqrt{n}} \sum_{j=0}^{m} \hZ_j}
\right)+o\left(\frac{n^\eps}{n} \right)
\\&
\overset{\eqref{Z-hZ2}}{=}e^{-\frac{\sigma^2 {\lambda}^2 n^\eps}{2 n}} \E\left(e^{i\frac{\lambda}{\sqrt{n}} \sum_{j=0}^{m} Z_j}
\right)+o\left(\frac{n^\eps}{n} \right).
\end{align}
Iterating this recurrence relation $m_n$ times we obtain
$$ \E\left(e^{i\frac{\lambda}{\sqrt{n}} \sum_{j=0}^{m_n} Z_m}\right)=e^{-\frac{\sigma^2 \lambda^2 }{2 }}+o(1) $$
completing the proof of part (a) of Theorem \ref{ThCLTabs}.

Part (b) can be established by a similar argument and we just briefly describe the necessary changes. 
Let $\E_n$ denote the expectation with respect to $\P_n,$ that  is,
$\E_n(\eta)=\E(\eta\rho_n).$ To extend the proof of the Central Limit Theorem to the setting of part (b) we need
to prove \eqref{smallblocks}  and \eqref{clt3} with $\E_n$ instead of $\E$. 
For \eqref{smallblocks}   we need to prove the analogues of \eqref{SmallBlocksA} and \eqref{truncated}.

We claim the following. First, \eqref{clt3} still holds with $\E_n$ instead of $\E.$ Second,
\begin{equation}
\label{SmallBlocksB}
\E_n \left(    \left(    \sum_{m=1}^{m_n} \ttZ_m  +\check{Z} \right)^2 \right) = O(n^{1-\eps+\eps^2}) =o(n). 
\end{equation}
(Note that in contrast to \eqref{SmallBlocksA} the sum here starts with $m=1,$ not $m=0.$)
Third, 
\begin{equation}
\label{TruncSame}
 \P_n\left(\sum_{l=n^\eps}^{n-1} \xi_l\neq \sum_{l=n^\eps}^{n-1} \xi_l^{K_n}\right)\to 0
\end{equation}
where $K_n=n^{\frac{1+\eps}{s}}.$ To prove \eqref{TruncSame} 
note that
$$ \P_n\left(\sum_{l=n^\eps}^{n-1} \xi_l\neq \sum_{l=n^\eps}^{n-1} \xi_l^{K_n}\right)
\leq n \max_l \P_n(\xi_l\neq \xi_l^{K_n}) $$
so \eqref{TruncSame} follows from (D3).
Observe that once these three points of the claim are established, the rest of the proof  of
part (b) proceeds exactly as in part (a). 
To obtain the the other two points of our claim we will need the following
\begin{lemm} \label{measurability} There exists a set $\bar G_m$ with
$\P(\bar G_m^c) \leq C \theta^{n^{\eps/100}}$
such that for $\omega \in \bar G_m$, for $l\geq n^\eps$ and for $\eta$ a bounded random variable we have that 
$$ |\E_n(\eta|\cF_l)(\omega)- \E(\eta|\cF_l)(\omega)|\leq C\|\eta\|_{\infty} \theta^{n^{\eps/100}}. $$
\end{lemm}

\begin{proof}  Let $\eta$ be a bounded random variable, $l\geq n^\eps$ and 
$\omega$ be such that 
\begin{equation}
\label{DenseBig}
\rho_{n,l}(\omega) \geq \theta^{l/2}, \quad |\E(\trho_{n,l}|\cF_l)(\omega)|\leq \theta^{2l/3} \quad \text{where} \quad 
\trho_{n,l}=\rho_n-\rho_{n,l}.
\end{equation}
Then
$$ \E_n(\eta|\cF_l)=
\frac{\E((\rho_{n,l}+\trho_{n,l}) \eta|\cF_l)}{\E((\rho_{n,l}+\trho_{n,l})|\cF_l)}=
\frac{\rho_{n,l}\E(\eta|\cF_l)+O(\theta^{2l/3})}{\rho_{n,l}+O(\theta^{2l/3})}=
\E(\eta|\cF_l)+O\left(\theta^{l/6}\right). $$
We prove now that the set where \eqref{DenseBig} fails has measure that is exponentially small in $l$. The proof consists of two steps. First, it follows form (D1) and (D2) that 
\begin{equation}
\label{BigOscDense}
 \P_n(|\trho_{n,l}|\geq \theta^l)\leq  C n^u \P(|\trho_{n,l}|\geq \theta^l)\leq C n^{2u} \theta^l 
\end{equation}
so 
$\P_n(\{|\E(\trho_{n,l}|\cF_l)(\omega)|\geq \theta^{2l/3}\})$ is exponentially small by Markov's inequality.
Second,
$$ \P_n(\rho_{n,l}<\theta^{l/2}, \trho_{n,l}<\theta^{l})\leq 
\P_n(\rho_{n}<2 \theta^{l/2})=
\E(\rho_n 1_{\rho_n<2\theta^{l/2}})\leq 2\theta^{l/2}$$
and so 
$ \P_n(\rho_{n,l}<\theta^{l/2})$  is exponentially small due to \eqref{BigOscDense}.
\end{proof} 

Lemma \ref{measurability}, together with Lemma \ref{lemma.clt3}, imply that  \eqref{clt3} holds for $\E_n$ instead of $\E$ provided that
we decrease slightly the set $\bG_m$ to $\bar G_m \cap \bG_m$. 

It remains to prove  \eqref{SmallBlocksB}. Note that the argument used to establish \eqref{ED} 
and \eqref{ExpMix} in fact gives for $\omega \in \bG_m$ 
$$ \E(\xi_m \xi_{m+j}|\cF_{n^{\eps/2}})(\omega)=
\begin{cases} \sigma_j+O(\theta^m)
& \text{ if }m>\frac{2uj}{\lvert\log_2 \theta\rvert} \\
O(\theta^j) & \text{otherwise.} \end{cases} $$
Hence Lemma  \ref{measurability} implies that for $ \omega \in \bar G_m \cap \bG_m$
$$ \E_n(\xi_m \xi_{m+j}|\cF_{n^{\eps/2}})(\omega)=
\begin{cases} \sigma_j+O(\theta^m)+O(\theta^{n^{\eps/200}}) 
& \text{ if }m>\frac{2uj}{\lvert\log_2 \theta\rvert} \\
O(\theta^j) +O(\theta^{n^{\eps/200}}) & \text{otherwise.} \end{cases} $$
This estimate implies \eqref{SmallBlocksB} by direct summation. 

The proof of Theorem \ref{ThCLTabs} is thus completed. \carre

\subsection{Bounded random variables.}
In case $\xi_l$ are bounded, one can take $\xi_l^K:=\xi_l,$
$b_{k, K}:=b_k$ so the conditions of Theorem \ref{ThCLTabs} simplify as follows.

\begin{enumerate}
\item[$\widetilde{(H2)}$] There exists filtration $\{\cF_l\}_{l\geq 0}$ 
such that for every $l,k$ there exists a bounded $\cF_{l+k}$-measurable random variable $\xi_{l,l+k}$ with $\E \xi_{l,l+k} = \E \xi_l$ such that 
$$\P( |\xi_{l}- \xi_{l,l+k}| \geq \theta^k) \leq C (l+1)^u \theta^k.$$

\item[$\widetilde{(H3)}$]  For $l,k$, there exists $G_{l,k}$ such that $\P(G_{l,k}^c) \leq C \theta^k$ and 
for $\omega\in G_{k,l}$ 
$$|\E (\hxi_{l+k} | \cF_l)(\omega)  |\leq C (l+1)^u  \theta^k . $$

\item[$\widetilde{(H4)}$] 
For $\omega\in G_{k,l}$  and $k'\geq k$
$$\left|\E (\hxi_{l+k}\hxi_{l+k'} | \cF_l )(\omega) - b_{k'-k}\right| \leq 
C(l+1)^u e^{u(k'-k)} \theta^{k}.$$
\end{enumerate} 

\begin{coro}
If $\xi_l$ is a bounded sequence satisfying $\widetilde{(H2)}$--$\widetilde{(H4)}$
then
$$\frac{\sum_{l=0}^{n-1} \left(\xi_l-\E(\xi_l)\right)}{\sqrt{n}} $$
converges as $n\to \infty$ to a normal distribution with zero mean and variance
$$\sigma^2=b_0+2\sum_{k=1}^\infty b_k. $$
\end{coro}

\section{Preliminaries on diagonal actions and Siegel transforms.}
\label{ScH-SL}

In this section we use the abstract Theorem \ref{ThCLTabs}
to prove Theorem  \ref{ThCLT-U-Lat}, \ref{ThCLT-V-Lat}, \ref{ThCLT-V-Lat-as}.
For this, we just have to check (H1)--(H4) for the case where our probability space is $\cM$ equipped with the Haar measure and
$\xi_l(\cL)=\Phi(g^l \cL),$ $\Phi=\cS(f)$, where $f\in C^{\bs, \br}(\R^{d+r})$ is a positive function supported on a compact set which does not contain $0.$ 

Before we construct the filtrations and prove (H1)--(H4) for the sequence $\xi_l(\cL)=\Phi(g^l \cL),$ we recall and prove preliminary results about functions defined on the space of lattices, on Siegel transforms, and on the action of diagonal matrices. We will cover this in Sections \ref{sec.rogers},  \ref{sec.truncation}, and \ref{sec.representative}
respectively.
Then we will prove Theorems \ref{ThCLT-U-Lat}, \ref{ThCLT-V-Lat}, \ref{ThCLT-V-Lat-as} in Section \ref{sec.proofs}. In Section \ref{sec.variances}, we compute the variances in the special case that interests us of $f$ being the characteristic function of $E_c$ given in \eqref{eq.Ec}. This will finish the proof of Theorems \ref{ThCLT-U}, \ref{ThCLT-V}, \ref{ThCLT-V-as}.

\subsection{Siegel transforms and Rogers' identities}
\label{sec.rogers}

\begin{prop}
\label{LmRogers}
\cite[Theorems 3.15 and 3.16]{M-npoint}, \cite[Appendix B]{EMV}
Suppose that $f\colon \cM \to \R$ is piecewise smooth function that is supported on a compact set which does not contain $0.$
Then,
\begin{align*} 
&(a)  \quad  \int_{\cM} \cS(f)(\cL) \; d\mu(\cL)=\int_{\R^{d+r}} f(\bx) d\bx; \\
&(b) \text{ \hspace{0.5cm} If $d+r>2$ then \hspace{0.8cm} }  \\
& \hspace{0.8cm}   \int_{\cM} \left[\cS(f)\right]^2 (\cL) \; d\mu(\cL)= \left[\int_{\R^{d+r}} f(\bx) d\bx\right]^2\\&\hspace{5cm} +\sum_{\overset{(p,q)\in \N^2}{gcd(p,q)=1}}
\int_{\R^{d+r}} \left[f(p\bx) f(q\bx)+f(p\bx) f(-q\bx)\right] dx;\\
\intertext{Suppose that $f\colon \tcM \to \R$ is piecewise smooth function of compact support. Then,}
&(c) \quad \int_{\tcM} \tcS(f)(\tcL) \; d\mu(\tcL)=\int_{\R^{d+r}} f(\bx) d\bx; \\
 &(d)\quad  \int_{\tcM} \left[\tcS(f)\right]^2 (\tcL) \; d\mu(\tcL)=
\left[\int_{\R^{d+r}} f(\bx) d\bx\right]^2+\int_{\R^{d+r}} f^2 (\bx) d\bx. 
\end{align*}
\end{prop}

\subsection{Rate of equidistribution of unipotent flows and representative partitions} \label{sec.representative}

Let $h_u$ be a one parameter subgroup of $\bbL.$ For example one can take the matrices with ones
on the diagonal, an arbitrary number in the upper left corner and zeros elsewhere. Note that
$$ g^n h_u \cL=h_{2^{d/r+1} u} g^n \cL. $$

The filtrations for which we will prove (H2)--(H4) for the sequence $\xi_l(\cL)=\Phi(g^l \cL),$ will consist of small arcs in the direction of the flow of $h_u$. The exponential mixing of the $\bbG$-action will underly the equidistribution and independence properties that are stated in (H1)--(H4). 

We will need the notion of representative partitions that was already used in \cite{DBpoisson}. 
These will be partitions  of $\cM$ whose elements are segments of $h_u$ orbits,
whose pushforwards by $g^l$
  will  become rapidly equidistributed. To guarantee the filtration property, we would ideally consider an increasing sequence of such partitions with pieces of size $2^{-l}$, $l=0,\ldots,n$. However, such partitions with fixed size pieces do not exist    because $h_u$ is weak mixing. We overcome this technical difficulty due to the following observations: 
\begin{enumerate}
\item Rudolph's Theorem (see \cite[Section 11.4]{CFS}) shows that for each $\breps$
we can find a partition $\cP$ into $h_u$-orbits such that the length of each element is
either $L$ or $L/\sqrt{2}$ and if $\cP_u=h_u(\cP)$ then
$$\mu(\cL:\; \exists u\in [0, L]:  \cP_u(\cL) \text{ has length } L/\sqrt{2})\leq \breps. $$
\item Given $n \in \N$, it suffices to check the properties (H2)--(H4) away from a set of 
measure less than $\theta^n$. 
\end{enumerate}

Having fixed $n$, we will therefore abuse notation and say that a partition is of size $L$ if $ \breps$ in 1) is less than 
$\th^n$. In light of this, let $\cP$ be  a partition  of size 1 and $\cP^l$ be its sub-partition of size $2^{-l}$. Due to (1) and (2) we can assume without loss of generality  that for every fixed $u \in [0,1]$, the partitions $\cP^l_u$
form an increasing sequence and that as a consequence the sequence $\cF_l$ of  $\sigma$-algebras  generated by $\cP^l_{u}$ forms a filtration.

Fix a small constant $\kappa>0$.
Given a collection $\Psi \subset C^{\bs,\br}(\cM),$  a set of  natural numbers $\{k_n\}_{n\in \N}$, and a number $L$,
we call a partition $\cP$ of size $L$ is \emph{representative}
with respect to 
$(\{k_n\}, \Psi)$ if for each $A\in \Psi$ and for each $n\in \N$,
$$ \mu\left(\cL: 
\left|\int_{g^{k_n} \cP(\cL)} A-\mu(A)\right|\geq  \| A \| _{C^{\bs, \br}(\cM)} \left(2^{k_n} L\right)^{-\kappa}\right) 
\leq   \| A \| _{C^{\bs, \br}(\cM)} \left(2^{k_n} L\right)^{-\kappa}. $$
The curve $g^{k_n} \cP(\cL)$ is of the form $h_u \bar{\cL}$ with $u \in [0,2^{k_n}L]$, and we use the notation $\int_{\gamma} A$ for the normalized integral $\frac{1}{2^{k_n}L}   \int_{0}^{2^{k_n}L} A(h_u  \bar{\cL}) du$.  
Given a finite collection $\Psi \subset C^{\bs,\br}(\cM)$,  $\{k_n\}_{n\in \N}$, and $L>0$, we let  
$$ \delta(\{k_n\}, \Psi, L):=
\sum_{A \in \Psi}   \| A \| _{C^{\bs, \br}(\cM)} \sum_n \left( 2^{k_n}L\right)^{-\frac{\kappa}{2}}. $$

Assume that $\delta \ll 1$. Then we have as in \cite[Proposition 7.1]{DBpoisson}

\begin{prop} \label{prop.rep} 
Let $\mathcal R(\{k_n\}, \Psi) \subset [0,1]$ be the set of $u$ such that $\cP_u$ is representative with respect to $(\{k_n\}, \Psi)$. Then  $\text{Leb}(\mathcal R(\{k_n\}, \Psi))\geq 1-\delta(\{k_n\}, \Psi, L)$.
\end{prop}

\begin{proof} We quickly recall how Proposition \ref{prop.rep} can be deduced, exactly as in \cite[Proposition 7.1]{DBpoisson}, from the polynomial mixing of the unipotent flow $h_u$. Indeed,  assuming that $\mu(A)=0$, polynomial mixing implies  
implies that
$$ |\mu(A(\cdot) A(h_u \cdot))|\leq C \cK_A^2 u^{-\kappa}$$
with $\cK_A =   \| A \| _{\bs,\br}.$ 
Thus for curves $\gamma(\bar{\cL})$ of the form $h_u {\bar{\cL}}$ with $u \in [0,L]$ we get that 
$$ \mu\left(\bar{\cL}:  \left| \int_{\gamma(\bar{\cL})} A\right|  \geq \cK_A L^{-\kappa_0}  \right)\leq C L^{-\kappa_0}$$
with $\kappa_0:=\kappa/3.$ This implies that if we consider a partition $\cP$ of size $L$ and its corresponding shifted partitions  $\cP_u$, $u\in [0,L]$ we get for the measure $\bar{\mu}=\mu \times \text{Leb}_{[0,1]}$ that 
$$ \bar{\mu}\left( (\cL,u)  \in \cM \times [0,1] : 
\left| \int_{\cP_u(\cL)} A \right| > \cK_A L^{-\kappa_0}\right)
\leq C L^{-\kappa_0}$$
where $\cP_u(\cL)$ denotes the piece of $\cP_u$ that goes through $\cL$. The claim of Proposition \ref{prop.rep} then follows by Markov's inequality. \end{proof}

\begin{rema}

Proposition \ref{prop.rep}  will be used in the next section to obtain a partition of $\cM$ into pieces of $h_u$ orbits satisfying the condition of Theorem \ref{ThCLTabs}. We could also use a partition into whole $\bbL$-orbits.
The proof of Proposition \ref{prop.rep} in that case would be simpler since we could use effective
equidistribution of horospherical subgroups \cite{KM1}. We prefer to use $h_u$ orbits instead since it allows us
to give unified proofs of Theorems \ref{ThCLT-U-Lat}, \ref{ThCLT-V-Lat}, and \ref{ThCLT-V-Lat-as}
(as well as Theorem  \ref{ThGenAction} in Section \ref{ScGenGroup}). 
\end{rema}


\subsection{Truncation of Siegel transforms.} \label{sec.truncation}

In this Section we give some useful results on truncations of a Siegel transform of 
a compactly supported function $f \in C^{\bs,\br}(\R^p)$. 
 These bounds are essential to control the truncated $\xi_n^K$ that appear in the abstract CLT of Section \ref{ScAbstract}.
We will leave all the proofs and constructions to Appendix A. In particular, we will define $\fh_{2, K} : \cM (\text{or } \tilde{\cM}) \to\R$, with 
the properties described in Lemma \ref{CrHs-Cs} below. 
We will always use the following notation for $\Phi=  \cS(f)$ (or  $\tcS(f)$) : $\Phi^K= \Phi  \fh_{2, K}.$  In the sequel we will consider $\xi_n^K:=\Phi^K\circ g^n$.  

\begin{lemm} 
\label{CrHs-Cs}
\cite{DBpoisson}
There exists a constant $Q>1$ such that
for each pair of integers $\bs, \br$ and each $R$ there is a constant $C=C(R, \bs, \br)$ 
such that the following holds.
Let $f$ be supported on $B(0, R)$ in $\R^p.$ 

(a) If $f\in C^{\bs}(\R^p)$ then
$$  \| \Phi^K  \| _{C^\bs (\cM)} \leq C K  \| f \| _{C^\bs(\R^p)}. $$

(b) If $f\in C^{\bs, \br}(\R^p)$ then
$$  \|  \Phi^K  \| _{C^{\bs, \br} (\cM)} \leq C K  \| f \| _{C^{\bs, \br}(\R^p)}. $$

(c) If $f\in C^{\bs, \br}(\R^p)$ then
$$  \|  \Phi^K \cdot \left(\Phi^K \circ g^j\right)
 \| _{C^{\bs, 2 \br} (\cM)} \leq C K^2  \| f \| _{C^{\bs, \br} (\R^p)}^2 Q^j. $$

\end{lemm}

\begin{lemm} \label{prop.ha1} For every $r,d$ there exists $C>0$ such that $\Phi=  \cS(f)$ or  $\Phi=\tcS(f)$ satisfy
\begin{align}
\label{prop.L1App}
 \E(\Phi-\Phi^K) &\leq \frac{C}{K^{d+r-1}}. \\
\label{prop.L2App}
 \E((\Phi-\Phi^K)^2) &\leq \frac{C}{K^{d+r-2}}. 
 \end{align}

 If $r=d=1$, then   $\Phi=\tcS(f)$ satisfies 
\begin{align}
\label{prop.L1App.rd1}
 \E(\Phi-\Phi^K) &\leq \frac{C}{K^{2}}. \\
\label{prop.L2App.rd1}
 \E((\Phi-\Phi^K)^2) &\leq \frac{C}{K}. 
 \end{align} 
In addition, the same inequalities \eqref{prop.L1App}--\eqref{prop.L2App.rd1} hold if the expectation is considered with respect to a measure that has a $C$-centrally smoothable density.
\end{lemm}

Recall that an $\bbL$-rectangle is a set of the form 
$\Pi(\bbR, \tcL)=\{\Lambda_\ba \tcL\}$ where $\ba$ belongs to the box $\bbR$ in $\R^{dr}.$
Also define $\fa \colon \cM\to \R$ by 
$$\fa(\cL)=\max\{(\text{covol}(\brcL))^{-1}: \brcL \le  \cL\}.$$

\begin{lemm} \label{truncation.xfix} For each $\breps, L$ there exists a constant $C>0$ such that  
for any box $\bbR$ whose sides are longer than $\breps$ and which is contained in $[-L, L]^{dr}$
and for any $\tcL$ with $\fa(\tcL)\leq L,$ $\Pi=\Pi(\bbR, \tcL)$ satisfies
\begin{align}
\label{prop.L0App.xfix}
 \P_\Pi (|\Phi \circ g^l|\geq K) &\leq \frac{C}{K^{d+r}}, \\
\label{prop.L1App.xfix}
\left| \E_{\Pi} (\Phi \circ g^l -\Phi^K \circ g^l ) \right| &\leq \frac{C}{K^{d+r-1}}, \\
\label{prop.L2App.xfix}
 \E_{\Pi}((\Phi \circ g^l  -\Phi^K\circ g^l )^2) &\leq \frac{C}{K^{d+r-2}},
 \end{align} 
 where $\P_\Pi$ is a restriction of the Haar measure on $\mu_{\bbL}$ to $\Pi$ and 
 $\E_{\Pi} $ is expectation with respect to $\P_\Pi.$
\end{lemm}


\section{Proof of the CLT for diagonal actions} \label{sec.proofs}

We are ready now to prove Theorem  \ref{ThCLT-U-Lat}, \ref{ThCLT-V-Lat}, \ref{ThCLT-V-Lat-as} using Theorem \ref{ThCLTabs}.

\subsection{CLT for lattices. Proof of Theorem  \ref{ThCLT-U-Lat}. }

For $f$ as in the statement of Theorem  \ref{ThCLT-U-Lat}, recall that we defined  $\Phi=\cS(f)$ and $\xi_l(\cL)=\Phi(g^l \cL).$ Recall also the notation $\Phi^K= \Phi  \fh_{2, K}$ and  $\xi_l^K:= \Phi^K \circ g^l$.
We will now prove (H1)--(H4) for the sequence $\{\xi_n\}$.

\subsubsection{Property (H1)}  Fix any $s \in (2,d+r)$. Property (H1) follows from 
inequalities \eqref{prop.L1App} and \eqref{prop.L2App} of Lemma \ref{prop.ha1}.
The fact that (H1) fails to hold when $d=r=1$ is the reason why the CLT does not hold in this case.

\subsubsection{Constructing filtrations.} We will use the notion of representative partitions of Section \ref{sec.representative} to construct the desired filtrations. 

First of all, note that to prove (H4) we need to deal with function of the form $\Phi^{K_n} \cdot \Phi^{K_n} \circ g^j$. Therefore we define for every $j \leq n$ the collection of functions 
$${\bf \Phi}^{(j)}:= \{\Phi, \Phi^{K_n}, \Phi^{K_n} \cdot \Phi^{K_n} \circ g^j\}$$
Fix constants $R_1\gg R_2\gg R_3\gg 1,$ and define for every $l \leq n$ the following collection of functions and sequences of integers 

\begin{equation}
\label{Collection1}
 \bigcup_{j \leq n}   \left(\{k+l\}_{k\geq R_3(\log_2 K_n+j)},  {\bf \Phi}^{(j)} \right) 
\end{equation}

Next let $\cP$ be a partition of size 1 and $\cP^l$ be its subpartition of size $2^{-l}.$ 
By Proposition \ref{prop.rep} 
and \ref{CrHs-Cs} there is $u$ such that for each $0\leq l\leq n$, 
$\cP^l_{u}$ is representative with respect to the collections of integers and functions in \eqref{Collection1}.
Let $\cF_l$ be the filtration of $\sigma$-algebras  generated by $\cP^l_{u}$.
Denote $\xi^K_{l, l+k}=\E(\xi_l^K|\cF_{k+l}).$

We claim that 
$(\xi_l^K, \{\cF_l\})$ satisfies (H1)--(H4)  with $u=2\bs$ provided that  $\theta$ is sufficiently close to 1.
Since (H1) has been checked above it remains to verify (H2)--(H4).

\subsubsection{Property (H2)}

If $k\leq C\log_2 K$ then (H2) holds if we take $u$ sufficiently large.
By Lemma \ref{CrHs-Cs} there are functions $\Phi^\pm$ such that
$$ \Phi^-\leq \Phi^K\leq \Phi^+, \quad 
 \| \Phi^+-\Phi^- \| _{L^1}\leq 2^{-\eps k},  \quad  \| \Phi^\pm \| _{C^1}\leq C K 2^{\eps \br k}. $$
Then
$$ \xi_l^--\xi_{l, k}^+\leq \xi_l^K-\xi_{l, k}^K \leq \xi_l^+-\xi_{l, k}^- $$
where $\xi_l^\pm$ and $\xi_{l, k}^\pm$ are defined analogously to $\xi_l^K$ and $\xi_{l, k}^K$ with
$\Phi^K$ replaced by $\Phi^\pm.$ Since $\Phi^\pm$ are Lipschitz, we have
$$ |\xi_l^\pm-\xi_{l, k}^\pm|\leq C K 2^{(\eps \br-1)k}. $$
So if $2^{\eps\br-1}\leq \theta^2 $ and $C$ is sufficiently large then 
$|\xi_l^K-\xi_{l, k}^K|\geq \theta^k $ implies 
$\xi_l^+-\xi_l^-\geq \frac{\theta^k}{3}. $
Hence Markov's inequality gives
$$ \P\left(\left|\xi_l^K-\xi_{l, k}^K\right|\geq \theta^k \right)\leq C \left(\frac{2^{-\eps}}{\theta}\right)^k . $$
This proves (H2) provided that $u$ is large enough and
$$ 2^{(\eps\br-1)/2} < \theta< 2^{-\eps} . $$

\subsubsection{Properties (H3) and (H4)}

(H3) follows from the definition of representative partition if $k\geq R_3\log_2 K_n$ while
for $k<R_3 \log_2 K_n$,
$$ \E(\xi_{k+l}^K|\cF_l)\leq K\leq K^u \theta^k $$
provided that $\theta$ is sufficiently close to 1.

Likewise, if $k>R_1(\log_2 K_n+k'-k)$
then (H4) holds by the definition of the representative partition with 
$$b_{K, k}=\E(\xi_k^K \xi_0^K)-(\E(\xi_0^K))^2. $$
If $k\leq R_1(\ln K_n+k'-k)$ we consider two cases

(a) $k'-k\leq R_2 \log_2 K_n$ and so $k< 2 R_1^2\log_2 K_n.$ In this case
(H4) trivially holds similarly to (H3). 

(b) $k'-k \geq R_2\log_2 K_n$ and so $k<2 R_1 (k'-k).$ Accordingly to establish (H4) with $b_{K, k}=0$ 
it suffices to show that there is a constant $\ttheta<1$ such that
\begin{equation}
\label{CVSmall}
\P\left(\left|\E (\hxi_{l+k}^K\hxi_{l+k'}^K | \cF_l )(\omega) \right| \geq \ttheta^j\right)\leq \ttheta^j. 
\end{equation}

We are going to show that \eqref{CVSmall} follows from already established (H1)--(H3). The argument is similar
to the proof of \eqref{MixK}. Namely, denoting by $j=k'-k$ we get
\begin{align*}\E(\hxi_{l+k'}^K \hxi_{l+k}^K|\cF_l)&=
\E(\hxi_{l+k+j}^K \hxi_{l+k, l+k+j/2}^K|\cF_l)+\E(\hxi_{l+k+j}^K (\hxi_{l+k,l+k+ j/2}^K-\hxi_{l+k}^K)|\cF_l)\\&=I+I\!\!I.\end{align*}
(H1a) and (H2) imply that
$ \P(|I\!\!I|\geq K\theta^j)\leq \theta^j.$
Next,
$$ |I|=\left|\E\left(\left[\hxi_{l+k, l+k+j/2}^K \E(\hxi_{l+k+j}^K |\cF_{l+k+j/2})\right]\Big|\cF_l\right)\right|$$
and (H3) shows that the expected value of the RHS is
$O(K^{2u} \theta^{j/2}) .$ Now Markov's inequality shows that $ \P(|I|\geq K^{2u} \theta^{j/4})\leq \theta^{j/4}.$
Combining the estimates of $I$ and $I\!\!I$ we obtain \eqref{CVSmall}.

Having checked (H1)--(H4), we have established Theorem  \ref{ThCLT-U-Lat}(a) via Theorem \ref{ThCLTabs}(a).  

\subsubsection{Starting from  localized initial conditions}
To prove Theorem  \ref{ThCLT-U-Lat}(b) we just need to check condition (D1)--(D3) of Theorem \ref{ThCLTabs}(b) for $\rho_N.$

Property (D1) follows from $(CN^u, \alpha)$-regularity since $ \| \rho_N \| _{L^\infty}\leq  \| \rho_N \| _{C^\alpha(\supp(\rho))}.$

To check (D2) let $\rho_{N,l}=\E(\rho_N|\cF_l).$ Then if
\begin{equation}
  \label{2-LFar}
h_{[-2^{-l}, 2^{-l}]} \cL\cap \partial(\supp(\rho))=\emptyset  
\end{equation}
then $\rho_N$ is Holder on the element of $\cP^l_u$ containing $\cL$ so
$$|\rho_N-\rho_{N,l}|\leq C N^u 2^{-\alpha l}. $$
so the exceptional set for (D2) consists of points violating \eqref{2-LFar}.
This set has a small measure since $\partial(\supp(\rho))$ is $(CN^u, \alpha)$-regular. 


Finally (D3) follows from inequalities \eqref{prop.L1App} and \eqref{prop.L2App} 
of Lemma \ref{prop.ha1} applied to the  centrally smoothable density $\rho_N$.

The proof  of Theorem  \ref{ThCLT-U-Lat} is thus complete. \carre

\subsection{CLT for affine lattices. Proof of Theorem  \ref{ThCLT-V-Lat}. }
For $f$ as in the statement of Theorem  \ref{ThCLT-V-Lat}  we define  $\Phi=\tcS(f)$ and $\xi_l(\tcL)=\Phi(g^l \tcL)$, with $\tcL \in \tcM$ distributed 
according to Haar measure. We also use  $\Phi^K= \Phi  \fh_{2, K}$ and  $\xi_l^K:= \Phi^K \circ g^l$.

If $(r,d)\neq (1,1)$ the analysis is exactly the same as in the case of linear lattices.

If $(r,d)=(1,1)$,  \eqref{prop.L1App} and \eqref{prop.L2App} of  Lemma \ref{prop.ha1} are not sufficient anymore to prove Property (H1), and we replace them by \eqref{prop.L1App.rd1} and \eqref{prop.L2App.rd1}. 
The proof of properties (H2)--(H4)  proceeds exactly as in the case of linear lattices.

Theorem  \ref{ThCLT-V-Lat}(a) thus follows from Theorem \ref{ThCLTabs}(a).

The changes needed for Theorem  \ref{ThCLT-V-Lat}(b) 
are the same as for
for Theorem  \ref{ThCLT-U-Lat}(b). Observe that (D3) and (H1) hold since
\eqref{prop.L1App}--\eqref{prop.L2App.rd1} are valid if Haar measure is replaced
with measures having centrally-smoothable densities. 
\carre

\subsection{Fixed $\bx.$ Proof of Theorem  \ref{ThCLT-V-Lat-as}. }
\label{ScasCLT} 

Here we deduce Theorem \ref{ThCLT-V-Lat-as} from a refined version of Theorem \ref{ThCLT-V-Lat}.
Using \eqref{prop.L0App.xfix} 
we conclude that it is sufficient to prove the Central Limit Theorem for sums with a shorter range of summation, 
$$ \frac{\sum_{n=N^\eps}^{N-1} \tcS(f)(g^n \tcL)-N\brf}{\tsigma \sqrt{N}}. $$
Next, take a large constant $\beta $ and for $\tcL\in \tcD$ let
$$ \cV(\tcL)=\{(D_\bt \bar{\Lambda}_\bb, (\by, 0))\tcL\}_
{|\bt|\leq N^{-\beta},  \; |\bb|\leq N^{-\beta},\; |\by|\leq N^{-\beta}
} . $$
Let 
$ \hcD=\bigcup_{\tcL\in \tcD} \cV(\tcL).$ We have a partition $\hPi$ of $\hcD$ into sets of the form
$\hPi(\tcL^*)=\bigcup_{\tcL\in \Pi(\tcL^*)} \cV(\tcL). $

\begin{lemm}
\label{LmFat}
If $\beta$ is sufficiently large large then there is a constant $\brdelta>0$ such that 
$$ \P_{\hPi(\tcL^*)} \left(\sum_{n=N^\eps}^{N-1} 
\left|\tcS(f)(g^n \tcL)-\tcS(f)(g^n (D_\bt \bar{\Lambda}_\bb, (\by, 0))\tcL)\right|\geq 1\right)\leq 
N^{-(1+\brdelta)} $$
except possibly for a set of $\tcL^*$ of measure $O(N^{-10}).$
\end{lemm}

\begin{proof}
First we use \eqref{prop.L0App.xfix} to replace $\tcS(f)$ by $\Phi^{K_N}.$ 
Hence denoting $\eps_N:=N^{-20}$ we get functions
$\Phi\pm$ such that
$$ \Phi^-\leq \Phi^{K_N}\leq \Phi^+, 
\quad  \| \Phi^+-\Phi^- \| \leq \eps_N,\quad\text{and}\quad 
 \| \Phi^\pm \| _{C^\bs}\leq C \eps_N^{-\br}. $$
Denoting $\hcL=(D_\bt \bar{\Lambda}_\bb, (\by, 0))\tcL$ we claim that
$$ \left|\Phi^{K_N}(g^n \tcL)-\Phi^{K_N}(g^n \hcL)\right|
\leq \left|\Phi^+(g^n \hcL)-\Phi^-(g^n \hcL)\right|+C K_N \eps_N^{-\br} N^{-\beta}=
I_n+I\!\!I_n. $$
Consider for example the case where $\Phi^{K_N}(g^n \hcL)\geq \Phi^{K_N}(g^n \tcL),$
the opposite case being similar. Then
$$0\leq \Phi^{K_N}(g^n \hcL)-\Phi^{K_N}(g^n \tcL)\leq 
\Phi^+(g^n \hcL)-\Phi^-(g^n \tcL)$$
$$\leq  
\Phi^+(g^n \hcL)-\Phi^-(g^n \hcL)+|\Phi^-(g^n \hcL)-\Phi^-(g^n \tcL)|.$$
The second term can be estimated by
$$ \| \Phi^- \| _{C^1} d(g^n \hcL, g^n \tcL)\leq C K_N \eps^{-\br} N^{-\beta} $$
proving our claim.
Next if $\beta$ is sufficiently large then 
$|\sum_n I\!\!I _n|\leq \frac{1}{2}$ while
$$\bigg\Vert\sum_n I_n\bigg\Vert_{L^1}\leq C\eps_N N. $$
Now the lemma follows from Markov's inequality.
\end{proof}
Lemma \ref{LmFat} allows to reduce Theorem \ref{ThCLT-V-Lat-as} to the following result.

\begin{theo} \label{ThCLT-V-Lat-as-thick} 
  Suppose that $(r,d)\neq (1,1).$
  For each $r>0$ and each $\eps>0$
the following holds.
If  $\P_{\tcL^*}$ denotes the uniform distribution on $\hPi(\tcL^*)$ then
$$\P\left(\tcL^*: \sup_z
\left|\P_{\cL^*}\left(\frac{\sum_{n=N^\eps}^{N-1} \tcS(f)(g^n \tcL)-N\brf}
     {\tsigma \sqrt{N}}\leq z\right)-
     \frac{1}{\sqrt{2\pi}}\int_{-\infty}^z e^{-s^2/2} ds\right|\geq \eps\right)
     =O\left(N^{-r}\right). 
     $$
\end{theo}

The proof of Theorem \ref{ThCLT-V-Lat-as-thick}  is also very similar to the proof of Theorem~\ref{ThCLT-U-Lat}. Let us describe the
necessary modifications. 

The property (H1) follows from Lemma \ref{truncation.xfix}  instead of  \eqref{prop.L1App} and \eqref{prop.L2App}.

To define the required filtration of (H2)--(H4) we need to adapt 
Proposition \ref{prop.rep} as follows. 
Take $\delta_N$ going to $0$ sufficiently slowly, for example, $\delta_N=1/N.$
We let $\cP$ be a partition into segments of $h_u$ orbits of size $\delta_N$ and $\cP^l$ the corresponding subpartitions of pieces with length $\delta_N 2^{-l}.$
We let $\cP^l_u$ be the translates by $h_u$ of these partitions and denote by  
$\cP^l_u(\tcL^*)$ the collection of pieces of   $\cP^l_u$ which are contained in $\hPi(\tcL^*).$
We say that 
$\tcL^*$ is $N$-good if  there exists $u\in [0, 1/N]$ such that
for each $N^\eps\leq l\leq N$,  $\cP^l_u(\tcL^*)$  is 
representative with respect to the families \eqref{Collection1}.
The proof of Proposition \ref{prop.rep} also shows the following.

\begin{lemm}
\label{LmMostGood}
Given $r\in \N$,  if we take $R_3$ in \eqref{Collection1} sufficiently large then
$$\P\left(\tcL^*\text{ is not $N$-good}\right)\leq \frac{C}{N^r}. $$
\end{lemm}
On the other hand if $\tcL^*$ is $N$-good then
the filtration generated by the partitions  $\cP^l_u(\tcL^*)$ 
satisfies (H2)-(H4). 
Theorem  \ref{ThCLT-V-Lat-as-thick} thus follows  from Theorem \ref{ThCLTabs}. \carre

\begin{rema}
\label{RkEff}
The argument given above does not tell us for which $\bx$ Theorem \ref{ThCLT-V-as} holds. 
Of course rational $\bx$ have to be excluded due to Theorem \ref{ThCLT-U}. Now a simple Baire category argument
shows that Theorem \ref{ThCLT-V-as} also fails for very Liouvillian $\bx.$ It is of interest to provide explicit
Diophantine conditions which are sufficient for Theorem \ref{ThCLT-V-as}. The papers
\cite{E, St} provide tools which may be useful in attacking this question. 
\end{rema}

\section{Related results}
\label{ScGenGroup}
The arguments of the previous section are by  no means limited to $SL_{d+r}(\R)/SL_{d+r}(\Z).$
In particular, we have the following result.

\begin{theo}
\label{ThGenAction}
Let $\cG$ 
be a $C^r$ diffeomorphism, $r\geq 2,$ of a manifold $\bbM$ and $\cH_u$ be a $C^r$
flow on that space. Suppose that both $\cG$ and $\cH_u$
preserve a probability measure $\mu$ and there exists $c>0$ such that 
$$ \cG^n \cH_u=\cH_{e^{c n} u} \cG^n. $$
Assume that $\cH$ is polynomially mixing, that is, there exist positive numbers
$\kappa$ and $K$ such that if
\begin{equation}
  \label{ACr}
A\in C^r(\bbM) \text{ and }\mu(A)=0 \text{ then}
\end{equation}
\begin{equation}
  \label{PolyMix}
\left|\mu(A(A\circ \cH_u))\right|\leq \frac{ K  \| A \| ^2_{C^r}}{u^\kappa}   
\end{equation}  

Fix $L>0$ and let $U$ be a random variable uniformly  distributed on $[0, L].$
Let $A$ satisfy \eqref{ACr}. Then for $\mu$ almost all $x\in \bbM$ 
$$ \frac{\sum_{n=0}^{N-1} A(\cG^n \cH_U x)}{\sqrt{N}} $$
converges as $N\to\infty$ to a normal random $\cZ$ variable with zero mean and variance
$$ \sigma^2=\sum_{n=-\infty}^\infty \int_\bbM A(x) A(\cG^n x) d\mu(x). $$
Moreover for each $\eps, r$ there is a constant $C$ such that
$$ \mu\left(x: \sup_{z\in\R} 
\left|\P\left(\frac{\sum_{n=0}^{N-1} A(\cG^n \cH_U x)}{\sqrt{N}} \leq z\right)-\P(\cZ\leq z)\right|>\eps\right)
\leq \frac{C}{N^r}. 
$$
The constant
$C$ can be chosen uniformly when $L$ varies over an interval 
$[\underline{L}, \overline{L}]$ for some $0<\underline{L}<\overline{L}.$
\end{theo}

We note that $\bbM$ need not be compact, so different $C^r$ norms on $\bbM$ need not be equivalent. In the Theorem \ref{ThGenAction}
above we assume that the compositions with $\cG$ preserve $C^r$ norm and moreover
\begin{equation}
  \| A\circ \cG^j \| _{C^r}\leq \brK Q^j  \| A \| _{C^r}
\text{ for some } \brK>0, \; Q>1. 
\end{equation}

The proof of Theorem \ref{ThGenAction} is similar to but easier than the proof of Theorem
\ref{ThCLT-V-Lat-as}. Namely since $A$ is bounded we only need to check conditions 
$\widetilde{(H2)}$--$\widetilde{(H4)}.$

Fix a partition $\Pi$ of $\bbM$ into $\cH_u$ orbit segments of size $L.$
$\Pi_{x}$ denote the partition of $\bbM$ of the form $\cH_{u(x)}\Pi$
which has $x$ as the boundary point
where $u(x)$ is the smallest positive number with this property.
As in Section \ref{ScasCLT} 
we let $\cP^l_x$ be the subpartition of $\Pi_x$ into segments of size
$\delta_n 2^{-l}.$

Consider the following collections. 
\begin{equation}
\label{Collection3}
\left(\{k+l\}_{k\geq R_1(\log_2 n+j)}, \left[ A \cdot \left(A \circ \cG^j \right)\right]\right) 
\end{equation}
and 
\begin{equation}
\label{Collection4}
(\{k+l\}_{k\geq R_3\log_2 n}, \{A\}). \end{equation}

We say that $x$ is $N$-good if
for each $l\leq N$ the partition $\cP^l_x$ is 
representative with respect to families \eqref{Collection3} and \eqref{Collection4}.
Lemma \ref{LmMostGood} is easily extended to show that for each $r$,
$$\P(x \text{ is not $N$-good})\leq \frac{C}{N^r}$$
provided that $R_1, R_3$ are large enough. Hence almost every $x$ is $N$-good for all sufficiently
large $N.$ Next 
let $\cF_l^{x}$ be the filtration
corresponding to $\cP^l_x.$ If 
$x$ is $N$-good then $\{\cF_l^{x}\}_{l\leq N}$ satisfies the conditions 
$\widetilde{(H2)}$-$\widetilde{(H4)}$ which implies the CLT in view of Theorem \ref{ThCLTabs}.

We note the following consequence of Theorem \ref{ThGenAction}.
\begin{cor}
\label{CorGenGroup}
Let $G$ be a semisimple Lie group without compact factors,
and $\Gamma\subset G$ be an irreducible
lattice. 
Let $h_u$ be a unipotent subgroup which is expanded by an element $g\in G$ in the sense that
$$ g^n h_u=h_{e^{c n} u} g^n$$
for some $c>0$. 
Fix $L>0$ and let $U$ be a random variable uniformly  distributed on $[0, L].$
Suppose that $r$ is large enough and let $A$ be a $C^r$ function of zero mean.
Then for Haar almost all $g_0\in G/\Gamma$ 
$$ \frac{\sum_{n=0}^{N-1} A(g^n h_U g_0)}{\sqrt{N}} $$
converges as $N\to\infty$ to a normal random $\cZ$ variable with zero mean and variance
$$ \sigma^2=\sum_{n=-\infty}^\infty \int_G A(g_0) A(g^n g_0) d\mu(g_0). $$
Moreover for each $\eps, r$ there is a constant $C$ such that
$$ \mu\left(g_0: \sup_{z\in\R} 
\left|\P\left(\frac{\sum_{n=0}^{N-1} A(g^n h_U g_0)}{\sqrt{N}} \leq z\right)-\P(\cZ\leq z)\right|>\eps\right)
\leq \frac{C}{N^r}. 
$$
The constant $C$ can be chosen uniformly when $L$ varies over an interval 
$[\underline{L}, \overline{L}]$ for some $0<\underline{L}<\overline{L}.$
\end{cor}
Corollary \ref{CorGenGroup} follows from Theorem \ref{ThGenAction} with $\bbM=G/\Gamma$ and $\cG$ and $\cH$ actions of
$g$ and $h$ respectively. The polynomial mixing \eqref{PolyMix} follows from \cite[Theorem 3.4]{KM2}.

\section{Variances.} \label{sec.variances}

\subsection{Variance of $U_N$}
\label{SSVarU}
Here we  establish \eqref{Sigma12}. We prove the formula for $\sigma_1^2,$ the computation for 
$\sigma_2^2$ is the same.
By \eqref{GreenKubo} it suffices to compute
$$ \sigma_1^2=\lim_{N\to\infty} \frac{1}{\ln N} 
\var\left(\sum_{n=0}^{[\log_2 N]-1} \cS(\one_{E_c}) (g^n \cL) \right) $$
where   the variance is taken with respect to the Haar measure on the space of lattices.
Note that 
$ \sum_{n=0}^{[\log_2 N]-1} \cS(\one_{E_c}) (g^n \cL)$ 
 can be replaced by
$ \cS(\one_{E_c(N)})(\cL)$
where
$$ E_c(N)=\{(x,y)\in \R^d\times \R^r: |x|\in [1, N], |x|{^{d/r}} y_j\in [0, c]\} . $$
Now Proposition \ref{LmRogers}(b) gives
\begin{align*}
\sigma_1^2 &=\lim_{N\to\infty} \frac{1}{\ln N} 
\sum_{gcd(p, q)=1} \int_{\R^{d+r}} \one_{E_c(N)} (px, py) \one_{E_c(N)}(qx, qy) dx dy 
\\
&=2 \lim_{N\to\infty} \frac{1}{\ln N} 
\sum_{\substack{gcd(p, q)=1 \\p<q}} \int_{\R^{d+r}} \one_{E_c(N)} (px, py) \one_{E_c(N)}(qx, qy) dx dy
\end{align*}
Since $p<q$, the last integral equals to 
\begin{align*}
\int_{\R^{d+r}} \one_{[1, N]}(|px|) \one_{[1, N]}(|qx|) 
\prod_j \left(\one_{[0, c]} (p^{1+d/r}|x|{^{d/r}} y_j) \one_{[0, c]} (q^{1+d/r}|x|{^{d/r}} y_j) 
\right) dx dy\\
=\int_{\R^{d+r}} \one_{{[1/p, N/q]}}(|x|) \prod_j \left(\one_{[0, c]} (q^{1+d/r}|x|{^{d/r}} y_j) 
\right) dx dy={\frac{c^r}{q^{d+r}}\; \int_{\R^d} \one_{{[1/p, N/q]}}(|x|) \frac{dx}{|x|^d} }
\end{align*}
{To evaluate the last integral we pass to the polar coordinates $x=\rho s$ where $s$ is a unit vector
in the Euclidean norm. Then,
$$\int_{\R^d} \one_{{[1/p, N/q]}}(|x|) \frac{dx}{|x|^d}=
\int_{\mathbb{S}^{d-1}} ds \int_{1/p|s|}^{N/q|s|} \frac{\rho^{d-1} d\rho}{|s|^d\rho^d}=
\ln \left(\frac{Np}{q}\right) \; \int_{\mathbb{S}^{d-1}} \frac{ds}{|s|^d}  .
$$
The second factor here equals to
$$ 
\int_{\mathbb{S}^{d-1}} \frac{ds}{|s|^d}=  d \int_{\mathbb{S}^{d-1}} ds \int_0^{1/|s|} \rho^{d-1} d\rho 
=d \int_{|x|<1} dx=d\; \mathrm{Vol}(\cB) .$$
}
Therefore,
$$\sigma_1^2 =2 c^r d 
\sum_{q=1}^\infty \frac{\varphi(q)}{q^{d+r}}\textrm{Vol}(\cB)
=
2 c^r d 
\frac{\zeta(d+r-1)}{\zeta(d+r)}\textrm{Vol}(\cB)$$
where the last step relies on \cite[Theorem 288]{HW}.

\subsection{Variance of $V_N$} 
Here we compute the limiting variance for $V_N.$ As in Section \ref{SSVarU} we consider the case
of boxes, the computations for balls being similar.

The same computation as in Section \ref{SSVarU} shows that we need to compute
$$ \lim_{N\to\infty} \frac{1}{\widehat{V}_N} 
\var\left(\cS(\one_{E_c(N)}) (\tcL) \right) $$
where the variance is taken with respect to the Haar measure on the space of affine lattices.
By Proposition \ref{LmRogers}(d) this variance equals to
$$ \int_{\R^{d+r}} \left[\one_{E_c(N)} \right]^2 (x, y) dx dy=
\int_{\R^{d+r}}\one_{E_c(N)} (x, y) dx dy=\widehat{V}_N. $$

\begin{rema}
The fact that the variance of $V_N$ has a simpler form than the variance of $U_N$ has the following 
explanation. Let 
$$\eta_k=\one_{B(k, d, r, c)} (k\ba), \quad \teta_k=\one_{B(k, d, r, c)} (\bx+k\ba)
$$ so that
$$ U_N=\sum_{|k|<N} \eta_k , \quad
V_N=\sum_{|k|<N} \teta_k .
$$
Then $\teta_k$'s are pairwise independent (even though triples 
$\teta_{k'}, \teta_{k''}, \teta_{k'''}$ are strongly dependent)
and hence uncorrelated (see e.g. \cite{Sch}) while $\eta_k$'s are not pairwise independent.
\end{rema}



\appendix 
\section{Truncation and norms}

For a fixed dimension $p \in \N$, we denote by $\cM$ the space of $p$ dimensional lattices. 
We let $C^\bs(\cM)$ denote  the space of smooth functions on $\cM$. Let $\fU_1, \fU_2, \dots ,\fU_{p^2-1}$ be a basis in the space of left invariant vector fields on $\cM.$ We let
$$  \| \Phi \| _{C^\bs}=\max_{0\leq k\leq \bs} \max_{i_1, i_2\dots i_k} 
\max_{\cL\in \cM} \left|\partial_{\fU_{i_1}} \partial_{\fU_{i_2}} \dots \partial_{\fU_{i_k}} \Phi(\cL) \right|.
$$
The space $C^\bs(\tcM)$ of smooth functions on the space of $r$-dimensional affine lattices is defined similarly.


We have the following inequality:
\begin{equation}
\label{Hs-CsProd}
 \| \Psi \Phi \| _{\bs}\leq C  \| \Psi \| _{C^\bs}  \| \Phi \| _{\bs}.
\end{equation}
Below we provide an extension to approximately smooth functions.

\begin{lemm}
\label{LmProd}
There is a constant $C$ such that if $\Phi_1, \Phi_2$ are $C^{\bs, \br}$ functions on $(\cM)$ or $\tcM$
then
$\Phi_1 \Phi_2$ is a $C^{\bs, 2\br}(\cM)$function and 
$$  \| \Phi_1 \Phi_2 \| _{C^{\bs, 2\br}} \leq C\;  \| \Phi_1 \| _{C^{\bs, \br}} 
\;  \| \Phi_2 \| _{C^{\bs, \br}}. $$
\end{lemm}

\begin{proof}
Without the loss of generality we may assume that 
$$ \| \Phi_1 \| _{C^{\bs, \br}}= \| \Phi_2 \| _{C^{\bs, \br}}=2. $$
Suppose first that 
$ 1\leq \Phi_j \leq 2.$ 
Given $\eps$ let $\Phi_j^\pm$ be the functions such that
$$ \Phi_j^-\leq \Phi_j \leq \Phi_j^+, \quad
 \| \Phi_j \| _{C^\bs} \leq 2 \eps^{-\br}, \quad 
 \| \Phi_j^+-\Phi_j^- \| _{L^1}\leq \eps. $$
Without the loss of generality we may assume that
$$0 \leq \Phi_j^-, \quad \Phi_j^+\leq 3, $$
since otherwise we can replace $\Phi_j^\pm$ by $\chi(\Phi_j^\pm)$ where $\chi$ is an appropriate
cutoff function. Then
$$ \Phi_1^- \Phi_2^-\leq \Phi_1 \Phi_2\leq \Phi_1^+ \Phi_2^+, \quad
 \| \Phi_1^+ \Phi_2^+-\Phi_1^- \Phi_2^- \| \leq 6 \eps, \quad
 \| \Phi_1^\pm \Phi_2^\pm \| _{C^\bs}\leq C \eps^{-2r}. $$
This proves the result in case $ 1\leq \Phi_j \leq 2.$ 
In general, write $\Phi_j= \tilde{\Phi}_j-\tilde{\tilde\Phi}_j$
where $\tilde\Phi_j=2 \| \Phi_j \| ,$ $\tilde{\tilde\Phi}_j=\tilde\Phi_j-\Phi_j$
and apply the foregoing argument to each term of the product
$ \big(\tilde{\Phi}_1-\tilde{\tilde\Phi}_1\big)
\big(\tilde{\Phi}_2-\tilde{\tilde\Phi}_2\big) . $
\end{proof}

Let
$$\fa(\cL)=\max\{(\text{covol}(\brcL))^{-1}: \brcL\subset \cL\}.$$
The role of this function is explained by the following lemmata.

\begin{lemm}
\label{LmST-CoVol}
For each sufficiently large $R$ there is a constant $C_1=C_1(R)$ such that if $f$ is supported on the 
ball of radius $R$ centered at the origin, then
\begin{equation}
\label{SiegelUpper1}
  \cS(f)(\cL)\leq C_1 \fa(\cL) 
\end{equation}
and
\begin{equation}
\label{SiegelUpper2}
  \tcS(f)(\cL+\bx)\leq C_1 \fa(\cL) .
\end{equation}
Also
\begin{equation}
\label{CoVolUpper}
\fa(\cL)\leq C_1 \cS(\one_{B(0, R)})(\cL) . 
\end{equation}
\end{lemm}

\begin{proof}

\eqref{SiegelUpper1} and \eqref{CoVolUpper}
are taken from (\cite[Lemma 5.1]{KSW}).
\eqref{SiegelUpper2} follows from \eqref{SiegelUpper1}. Indeed suppose that
$\tcS(f)(\cL+\bx)\neq 0.$ Then there exists 
$\bre\in \cL+\bx$ such that $f(\bre)\neq 0.$ 
Now we have
$$ \tcS(f)(\cL+\bx)=\cS(\tau_\bre f)(\cL) $$
where $\tau_\bre(f)(e)=f(e+\bre).$ Applying \eqref{SiegelUpper1}  to $\tau_\bre(f)$ we get 
\eqref{SiegelUpper2}.  \end{proof}

\begin{lemm}
\label{LmCoVolTail}
There is a constant $C_2$ such that 
$$ \mu(\cL: \fa(\cL)>t)\leq \frac{C_2}{t^{d+r}}. $$
\end{lemm}

\begin{proof} The proof follows
from \eqref{CoVolUpper} and the estimate
$$ \mu(\cL: \cS(\one_{B(0, R)})>t)\leq \frac{C_2}{t^{d+r}}$$
given in \cite[Theorem 4.5]{M-npoint}. \end{proof}

\begin{lemm}
\label{LmCutOff}
\cite{KM2, DBpoisson}
For each $\bs$ there are constants $C_3, C_4$ such that for each $K \geq 1$ there is a function
$\fh_{1, K}:\cM\to\R$ such that
\begin{itemize}
\item[(C1)] $0\leq \fh_{1, K}\leq 1,$
\item[(C2)] $\fh_{1,K}(\cL)=1$  if $\fa(\cL)\geq K,$
\item[(C3)] $\fh_{1,K}(\cL)=0$ if $\fa(\cL)\leq C_3 K,$
\item[(C4)] $ \| \fh_{1,K} \| _{C^\bs(\cM)} \leq C_4.$
\end{itemize}
\end{lemm}
For example, one can take
$$ \fh_{1,K}=\int_{SL_{p}(\R)} G(g) \one_{\fa(gL)>C_3 K}(g\cL)\; d\mu(g), \quad $$
where $G$ is a non negative function with integral one supported on the set $C_3^{-1}\leq  \| g \| \leq C_3.$
We write $\fh_2=1-\fh_1.$ We can also regard $\fh_j$ as functions on $\tcM$ defined by the formula
$\fh_j(\cL, \bx)=\fh_j(\cL).$
We are ready now to give the proofs of the statements from Section \ref{sec.truncation}.

\begin{proof}[Proof of Lemma \ref{CrHs-Cs}.] We prove the estimates for $\cS,$ the estimates for $\tcS$ are similar.

(a) We have
$$ |\cS(f) (\cL)| \fh_{2, K}(\cL) \leq \one_{\fa(\cL)\leq K} |\cS(f)|(\cL)\leq C K $$
where the last step uses Lemma \ref{LmST-CoVol}(a).
The derivatives of $\cS(f)$ are estimated similarly using the formula
$$ \partial_\fU(\cS(f))=\cS(\partial_\brfU f) 
\text{ where } (\partial_\brfU f)(\bx)=\frac{d}{dt}\big|_{t=0} f(e^{t\fU} \bx). $$

(b) Given $\eps$ consider functions $f^\pm$ such that $f^-\leq f\leq f^+$ and \eqref{Cs-L1} holds.
Then 
$\cS(f^-)\fh_{2, K}\leq \cS(f)\fh_{2, K} \leq \cS(f^+) \fh_{2, K}$ so the result follows from
already established part (a) and Proposition \ref{LmRogers}(a).

(c) We already know from part (b) that
$$ \| \cS(f) \fh_{2, K} \| _{C^{\bs, \br}}=O(K). $$
A similar argument shows that
$$ \| \left[\cS(f) \fh_{2, K}\right]\circ g^j  \| _{C^{\bs, \br}}=O(2^{j \bs} K). $$
Now the result follows by Lemma \ref{LmProd}.
\end{proof}

\begin{proof}[Proof of Lemma \ref{prop.ha1}.]
Property (C2) of Lemma \ref{LmCutOff} and Lemma \ref{LmCoVolTail} imply that
\begin{align*}
 \left|\E(\Phi-\Phi^K)\right| & \leq  C \int_{\fa(\cL) \geq K/C_3} \left|\Phi(\cL) \right| d\mu \\
&\leq C \int_{\fa(\cL) \geq K/C_3} \fa(\cL) d\mu \\
&\leq \frac{C}{K^{d+r-1}} 
 \end{align*} 
 which gives \eqref{prop.L1App}, and
\begin{align*}
 \E((\Phi-\Phi^K)^2) &\leq C \int_{\fa(\cL) \geq K/C_3} \Phi^2 (\cL) d\mu \\
&\leq C  \int_{\fa(\cL) \geq K/C_3} \fa^2 (\cL) d\mu
\\
&
\leq \frac{C}{K^{d+r-2}} 
 \end{align*} 
which gives \eqref{prop.L2App}.

Now we deal with the case of affine lattices and $(r,d)=(1,1)$. Let $\cL$ be such that $\fa(\cL)=t\gg 1.$
We claim that
\begin{equation}
\label{Impr1+1}
\int_{\R^2/\cL} \Phi(\cL, \bx) d\bx\leq C , \quad \int_{\R^2/\cL} \Phi^2(\cL, \bx) d\bx\leq C t. 
\end{equation}
This gives the required improvement of an extra power of $t$ 
that is sufficient to verify  \eqref{prop.L1App.rd1} and \eqref{prop.L2App.rd1} using Lemma \ref{LmCoVolTail}.

To show \eqref{Impr1+1} let $e_1$ be the shortest vector in $\cL.$ Note that $|e_1|$ is of order $1/t.$
Thus $\cL$ is contained in a union of lines going in the direction 
of $e_1$ so that the distance between the lines is almost $t.$ If we shift $\bx$ in the direction perpendicular
to $e_1$ then the probability that one of the shifted lines intersects the ball of a fixed radius around the origin
is $O(1/t).$ Since $\Phi(\cL, \bx)=O(t)$ due to \eqref{SiegelUpper2}, the estimate  \eqref{Impr1+1} follows.

We now show that  \eqref{prop.L1App} and \eqref{prop.L2App} hold if Haar measure is
replaced by a measure having  a $C$-centrally smoothable density with respect to Haar measure.
We just prove \eqref{prop.L1App} in the lattice case, since the proofs  of \eqref{prop.L2App}
as well as the proofs for affine lattices are  exactly the same. 
\begin{align*} \left| \E_\rho(\Phi-\Phi^K)\right| &= \int_\cM 1_{\Phi(g^l \cL)> K}  \left|\Phi(g^l \cL)\right| \rho (\cL)d\mu \\
&\leq  C \int_\cM 1_{\fa(g^l \cL)> K/C}  {\fa(g^l \cL)} \rho(\cL) d\mu \end{align*}
where the inequality follows from  Lemma \ref{LmST-CoVol} since $\Phi=\cS(f)$ with $f$ having a compact support.
Next, the $K$-central smoothability of $\rho$  and Lemma \ref{LmCoVolTail} imply that

\begin{align*}
\int_\cM      1_{\fa(g_l \cL)> K/C}   {\fa(g_l \cL)} \rho (\cL) d\mu
&
=\int_{\cA} \phi(a) \int_\cM     1_{\fa(g_l a \cL)> K/C} \rho(a \cL)  {\fa(g_l(a \cL))} d\mu da \\
&
\leq 
C \int_\cM 1_{\fa(g_l \cL)> K/C} {\fa(g_l \cL)} \left(\int_{\cA} \phi(a) \rho(a \cL) da \right) d\mu 
\\
&
\leq C \int_\cM 1_{\fa(g_l \cL)>K/C} \fa(g_l \cL) d\mu \\
&
\leq C \int_{\fa(\cL) > K/C} \fa(\cL) d\mu 
\\
&
\leq C{K^{1-(d+r)}} . \end{align*}
Inequality \eqref{prop.L1App} is thus proved.

As for  the case of affine lattices and $(r,d)=(1,1)$,  \eqref{prop.L1App.rd1} and \eqref{prop.L2App.rd1} can be proved as in the case of Haar measure, if one makes the following two observations:
\begin{enumerate}
\item  Equation \eqref{Impr1+1} still holds for a measure with density $\rho;$
\item The tail estimate of Lemma \ref{LmCoVolTail} can be proved for measures with centrally smoothable
  densities following the same lines as the proof of \eqref{prop.L1App}. 
\end{enumerate}
\end{proof}


\begin{proof}[Proof of Lemma \ref{truncation.xfix}.]
 Let $U$ be the set of points obtained by issuing 
local center-stable manifolds through all points of $\Pi.$ 
That is 
$$U=\bigcup_{\tcL'\in \Pi, |\sigma|\leq 1, |\bb|\leq 1} 
D_\bt \bar{\Lambda}_\bb \tcL' . $$

Let $\mu_{\bbL}$ and $\mu_{\bbG}$
denote the Haar measures on $\bbL$ and $\bbG$ respectively.
Then 
\begin{align*}\P_\Pi(|\Phi(g^l \tcL')|>K)
&
\leq C_1 \mu_\bbG(\hcL\in U: |\fa(g^l \hcL)|>K)
\\
&\leq
C_2 \mu_\bbG(\hcL\in \tcM: |\fa(g^l \hcL)|>K)
\\
&
\leq \frac{C}{K^{d+r}}. 
\end{align*}
Similarly for $q\in\{1, 2\}$,
\begin{align*}
\int_{\Pi} |\Phi_l(\tcL')-\Phi_l^K(\tcL')|^q \; d\mu_{\bbL}(\tcL')
&\leq
C \int_{\Pi} |\Phi(g^l \tcL')|^q \one_{\fa(g^l \tcL')>K} \; d\mu_{\bbL}(\tcL')
\\
&\leq 
C \int_{U} |\fa(g^l \tcL')|^q \one_{\fa(g^l \tcL')>K} \; d\mu_{\bbG}(\tcL') \end{align*}
where the last step follows from \eqref{SiegelUpper2}.
Since the integrand depends only on projection of $\tcL'$ to $\cM$ the integral can be estimated by
$$C \int_{\cM} |\fa(g^l \cL)|^q \one_{\fa(g^l \cL)>K} \; d\mu_{\bbG}(\cL)=
C \int_{\cM} |\fa(\cL)|^q \one_{\fa(\cL)>K} \; d\mu_{\bbG}(\cL). $$
Thus \eqref{prop.L1App.xfix} and  \eqref{prop.L2App.xfix} follow 
Lemma \ref{LmCoVolTail}.
\end{proof}

\end{document}